\documentclass[11pt,a4paper]{amsart}
\usepackage[utf8x]{inputenc}
\usepackage{ucs}
\usepackage{amsmath}
\usepackage{amsfonts}
\usepackage{amsthm}
\usepackage{amssymb}
\usepackage{geometry}
\usepackage{graphicx}
\usepackage{color}
\usepackage{dsfont}
\usepackage{stackrel}
\usepackage{enumerate}
\usepackage{stmaryrd}
\usepackage[hidelinks]{hyperref}

\newtheorem{theo}{Theorem}
\newtheorem{lem}{Lemma}
\newtheorem{prop}{Proposition}

\newtheorem{rem}{Remark}

\newcommand{\E}{\mathbb{E}}

\newcommand\blfootnote[1]{%
  \begingroup
  \renewcommand\thefootnote{}\footnote{#1}%
  \addtocounter{footnote}{-1}%
  \endgroup
}

\title[Fluctuations in Salem--Zygmund almost sure central limit theorem]{Fluctuations in Salem--Zygmund \\ almost sure central limit theorem}

\author{J\"urgen Angst}
\address{Univ Rennes, CNRS, IRMAR - UMR 6625, F-35000 Rennes, France}
\email{jurgen.angst@univ-rennes1.fr}

\author{Guillaume Poly}
\address{Univ Rennes, CNRS, IRMAR - UMR 6625, F-35000 Rennes, France}
\email{guillaume.poly@univ-rennes1.fr}


\begin{document}

\begin{abstract}
Let us consider two sequences of independent and identically distributed random variables $(a_k)_{k \geq 1}$ and $(b_k)_{k \geq 1}$ defined on a probability space $(\Omega, \mathcal F, \mathbb P)$, following a symmetric Rademacher distribution and the associated random trigonometric polynomials $S_n(\theta)= \frac{1}{\sqrt{n}} \sum_{k=1}^n a_k \cos(k\theta)+b_k \sin(k\theta)$. A seminal result \cite[Theorem 3.1.1]{salem1954} by Salem and Zygmund, ensures that $\mathbb{P}-$almost surely, $\forall t\in\mathbb{R}$
\begin{equation}\label{SZAbs-CLT}
\lim_{n \to +\infty} \frac{1}{2\pi}\int_0^{2\pi} e^{i t S_n(\theta)}d\theta=e^{-t^2/2}.
\end{equation}
This result was then further generalized in various directions regarding for instance the coefficients distribution, their structure of dependency, the dimension and the nature of the ambient manifold, or else the supremum norm of $S_n(\cdot)$. One may consult \cite{AP2021,benatar2020moments,gass2020almost,kahane1993some,weber2006stronger} and the references therein. To the best of our knowledge, the natural question of the fluctuations in the limit \eqref{SZAbs-CLT} has not been tackled so far and is precisely the object of this article. Namely, for general i.i.d. symmetric random coefficients having a finite sixth-moment and for a large class of continuous test functions $\phi$ we prove that
\[
\sqrt{n}\left(\frac{1}{2\pi}\int_0^{2\pi} \phi(S_n(\theta))d\theta-\int_{\mathbb{R}}\phi(t)\frac{e^{-\frac{t^2}{2}}dt}{\sqrt{2\pi}}\right)\xrightarrow[n\to\infty]{\text{Law}}~\mathcal{N}\left(0,\sigma_{\phi}^2+\frac{c_2(\phi)^2}{2}\left(\mathbb{E}(a_1^4)-3\right)\right).
\]
Here, the constant $\sigma_{\phi}^2$ is explicit and corresponds to the limit variance in the case of Gaussian coefficients and $c_2(\phi)$ is the coefficient of order $2$ in the decomposition of $\phi$ in the Hermite polynomial basis. Surprisingly, it thus turns out that the fluctuations are not universal since they both involve the kurtosis of the coefficients and the second coefficient of $\phi$ in the Hermite basis.\par
\medskip
Asides from the result which has its own interest, the method we develop is robust and enables one to use Wiener chaotic expansions for proving central limit Theorems for general random fields which are not restricted to be functionals of an underlying Gaussian field and thus it opens the door to many other applications, especially regarding central fluctuations for nodal functionals for non-Gaussian random fields. The cornerstone of our method gathers several seminal results of the so-called Malliavin--Stein method \cite{nualart2005central,peccati2005gaussian,nourdin2009stein,mossel2010noise,nourdin2010invariance} with algebraic considerations around Newton sums and the Newton--Girard formula. It is not based on the moment method but instead exploits some invariance principles which roughly speaking enable to extend CLTs for homogeneous sums in Gaussian variables to the case of general random variables.
\end{abstract}

\if{Abstract arXiv
Let us consider i.i.d. random variables $\{a_k,b_k\}_{k \geq 1}$ defined on a common probability space $(\Omega, \mathcal F, \mathbb P)$, following a symmetric Rademacher distribution and the associated random trigonometric polynomials $S_n(\theta)= \frac{1}{\sqrt{n}} \sum_{k=1}^n a_k \cos(k\theta)+b_k \sin(k\theta)$. A seminal result by Salem and Zygmund ensures that $\mathbb{P}-$almost surely, $\forall t\in\mathbb{R}$
\[
\lim_{n \to +\infty} \frac{1}{2\pi}\int_0^{2\pi} e^{i t S_n(\theta)}d\theta=e^{-t^2/2}.
\]
This result was then further generalized in various directions regarding the coefficients distribution, their dependency structure or else the dimension and the ambient manifold. To the best of our knowledge, the natural question of the fluctuations in the above limit has not been tackled so far and is precisely the object of this article. Namely, for general i.i.d. symmetric random coefficients having a finite sixth-moment and for a large class of continuous test functions $\phi$ we prove that
\[
\sqrt{n}\left(\frac{1}{2\pi}\int_0^{2\pi} \phi(S_n(\theta))d\theta-\int_{\mathbb{R}}\phi(t)\frac{e^{-\frac{t^2}{2}}dt}{\sqrt{2\pi}}\right)\xrightarrow[n\to\infty]{\text{Law}}~\mathcal{N}\left(0,\sigma_{\phi}^2+\frac{c_2(\phi)^2}{2}\left(\mathbb{E}(a_1^4)-3\right)\right).
\]
Here, the constant $\sigma_{\phi}^2$ is explicit and corresponds to the limit variance in the case of Gaussian coefficients and $c_2(\phi)$ is the coefficient of order $2$ in the decomposition of $\phi$ in the Hermite polynomial basis. Surprisingly, it thus turns out that the fluctuations are not universal since they both involve the kurtosis of the coefficients and the second coefficient of $\phi$ in the Hermite basis.
\end{altabstract}
}\fi

\maketitle

\textcolor{white}{blanc}
\blfootnote{Univ Rennes, CNRS, IRMAR - UMR 6625, F-35000 Rennes, France.}
\blfootnote{This work was supported by the ANR grant UNIRANDOM, ANR-17-CE40-0008.}

\section{Introduction and statement of the results}
Let us first introduce our basic notations, motivate the problematic in consideration and state our main results.
\subsection{Setting and motivations}\label{setting-gaussien}
On a suitable probability space $(\Omega,\mathcal{F},\mathbb{P})$, we consider a family of real random variables $\{a_k,b_k\}_{k\ge 1}$  that are all independent and identically distributed, their common distribution being symmetric with unit variance. We then consider the product probability space 
\[
\left(\Omega,\mathcal{F},\mathbb{P}\right)\otimes\left([0,2\pi],\mathcal{B}\left([0,2\pi]\right),\frac{dx}{2\pi}\right)
\]
on which we define the new random variable $X:(\omega ,x)\in\Omega \times[0,2\pi]\mapsto x$, whose law $\mathbb P_X$ is then naturally uniformly distributed on $[0,2\pi]$, i.e.  $\mathbb{P}_X=\frac{dx}{2\pi}$, and by construction $X$ is independent of the whole sequence $\{a_k,b_k\}_{k\ge 1}$. Besides, for any random variable $Z\in L^1\left(\mathbb{P}\otimes \mathbb{P}_X\right)$ we set
\begin{eqnarray*}
\mathbb E[Z]&=&\mathbb{E}_{\mathbb{P}}\left[Z\right]:=\int_{\Omega_1}Z(\omega, \cdot ) d\mathbb{P}(\omega)\in L^1\left(\frac{dx}{2\pi}\right),\\
\mathbb{E}_{X}\left[Z\right]&:=&\frac{1}{2\pi}\int_0^{2\pi}Z(\cdot, x) dx\in L^1(\mathbb{P}).
\end{eqnarray*}

\noindent
For $x \in \mathbb R$, we then define:
\begin{equation}\label{defi-S_n}
S_n(x):=\frac{1}{\sqrt{n}}\sum_{k=1}^n a_k \cos(kx)+b_k\sin(kx).
\end{equation}

In this setting, the celebrated almost sure central limit Theorem established by Salem and Zygmund  in \cite{salem1954} asserts that $\mathbb P-$almost surely, for any $t \in \mathbb R$, as $n$ goes to infinity 
\[
\mathbb E_X \left[ e^{i t S_n(X)}\right] \xrightarrow[n\to\infty]~ e^{-t^2/2}.
\]
In other words, $\mathbb P-$almost surely, as $n$ goes to infinity, $S_n(X)$ converges in distribution under $\mathbb P_X$ towards a standard Gaussian $\mathcal N(0,1)$.
Equivalently, if we denote by 
\[
\gamma(dx):=e^{-\frac{x^2}{2}}\frac{dx}{\sqrt{2\pi}}
\]
the standard Gaussian distribution, then for any continuous and bounded function $\phi$ we have $\mathbb{P}-$almost surely that

\begin{equation}\label{SZ-CLT}
\mathbb{E}_X\left[\phi\left(S_n(X)\right)\right])\xrightarrow[n\to\infty]~ \gamma(\phi):=\int_{\mathbb{R}}\phi(x)\gamma(dx).
\end{equation}

Our goal in this article is to address the natural question of the fluctuations in the limit \eqref{SZ-CLT} above. In particular, one would like to make precise the order of magnitude of the convergence and after a proper normalization, exhibit a non-degenerate limit in distribution. Among the natural related questions, one would like to know if the latter limit in distribution is Gaussian or not, and if and how it depends on the particular common law of the coefficients $\{a_k,b_k\}_{k\ge 1}$, that is to say determining whether the fluctuations are universal or not in Salem--Zygmund CLT.
\par
\medskip
Our initial motivation is a better understanding of the almost sure behavior of random trigonometric polynomials initiated in our recent work \cite{AP2021}, where we quantified the original CLT by Salem--Zygmund in some appropriate metrics, and moreover showed that it can be extended to a functional setting.  The underlying model of random trigonometric polynomials is of course interesting in itself but it can also be seen as a toy model for the more difficult study of random linear combinations of Laplace eigenfunctions on Riemannian manifolds where similar Salem--Zygmund type results can be established. One may see for instance \cite{gass2020almost} and the references therein.
\par
\medskip
Furthermore, the use of Salem--Zygmund like CLTs has recently proven its efficiency in order to establish the almost sure asymptotics for the number of zeros or nodal volume of different models of random functions \cite{AP2021, APP21, gass2020almost}. Indeed, as mentioned above, the almost sure central limit of Salem--Zygmund can be extended to a functional setting whereas in the other hand, the number of zeros and the nodal volume can both be expressed as expectations of certain functionals of the random field with respect to some random variable $X$ uniformly distributed on the ambient manifold. Such representation formulae for the nodal volume coupled with the knowledge of the limit in distribution of the random field with respect to $\mathbb{P}_X$ offer an efficient method to study the almost sure asymptotics of the nodal volumes. Let us mention that somehow philosophically related approaches, but with a strong arithmetic flavor, have been used in the setting of Laplace eigenfunctions on the torus via the so-called Bourgain de-randomization method, see \cite{buckley2016number,bourgain2014toral}. One difference between these ideas and our approach is that the genericity of the eigenfunctions is there mainly related to the generic arithmetic properties of the eigenvalues, whereas in our setting the genericity is exclusively related to the randomness of the coefficients.

\medskip

The study of the fluctuations in Salem--Zygmund CLT and in particular the fact the the limit depends or not on the particular common law of the coefficients $\{a_k,b_k\}_{k\ge 1}$ can be seen as a first step towards a better understanding of the universality phenomenon for trigonometric models, which is the object of a vast literature, both at a microscopic and a macroscopic scales, see e.g. \cite{flasche,iksanov2016,nousAMS,MR3846831,nguyen2017roots} and the references therein.
\par
\medskip
As we shall see below, the fluctuations in Salem--Zygmund CLT do depend on the distribution of  the coefficients, via their fourth moment. The non-universality of the fluctuations we establish here could thus be seen as a partial explanation to the surprising phenomenon that, for random trigonometric polynomials whose degree tends to infinity, the first order asymptotics -- i.e. the expectation -- of the number of zeros, is generally universal \cite{flasche, APP21}, whereas the second order asymptotics -- i.e. the variance -- is not, as observed in the references \cite{Bally2018,do2019random}. This situation is somehow typical of random trigonometric polynomials models as for other models of random polynomials such that the celebrated Kac polynomials, the fluctuations are indeed universal as they do not depend on the specific distributions of the coefficients $\{a_k,b_k\}_{k\ge 1}$.The reader can consult for example the references \cite{maslova1974variance,nguyen2019random} for more details.

\newpage
\subsection{Statement of the results}
In order to state precisely our mains results, let us briefly recall basic concepts around Hermite polynomials. For any integer $p\ge 0$, we shall denote by $H_p$ the $p$-th Hermite polynomial which is classically defined by 
\begin{equation*}
H_q(x)=(-1)^q e^{\frac{x^2}{2}} \frac{d^q}{dx^q}\left[e^{-\frac{x^2}{2}}\right].
\end{equation*}

It is well known that Hermite polynomials form a complete orthogonal system in $L^2\left(\gamma\right)$ so that any square integrable function $\phi$ can be expanded in the following way $\phi(x)=\sum_{k=0}^\infty c_k(\phi) H_k(x)$, where the sum converges in $L^2(\gamma)$ and where the $c_k(\phi)$ denote the so-called Hermite coefficients of $\phi$. As, $\gamma(H_k^2)=k!$, by orthogonality, one obtains in particular that 
\[
\gamma(\phi^2)=\int_{\mathbb{R}}\phi^2(x)d\gamma(x)=\sum_{k=0}^\infty c_k^2(\phi) k! .
\]
The first main result of the article describes the fluctuations in Salem--Zygmund CLT in the particular case where the coefficients are i.i.d. standard Gaussian variables. 
\begin{theo}\label{Main-Gaussian}
Suppose that the variables $(a_k)$ and $(b_k)$ are all independent with standard Gaussian distribution and that $\phi \in L^2(\gamma)$. Then as, $n$ goes to infinity, we have the convergence in distribution
\begin{equation}\label{Main-Gaussian-eq}
\sqrt{n}\left(\mathbb{E}_X\left[\phi\left(S_n(X)\right)\right]-\gamma(\phi) \right)\xrightarrow[n\to\infty]{\text{Law}}~\mathcal{N}(0,\sigma_\phi^2),
\end{equation}
where the variance of the limit Gaussian distribution is given by
\begin{equation}\label{variance-infinie}
\sigma_\phi^2:=\sum_{k={}{2}}^\infty c_k^2(\phi) {}{\frac{k!}{2\pi} }\int_{\mathbb{R}}\left(\frac{\sin(x)}{x}\right)^k dx<+\infty.
\end{equation}
\end{theo}
\begin{rem}\label{rem.first}
The above theorem shows that the exact rate of convergence in Salem--Zygmund CLT is of order $1/\sqrt{n}$, which was indeed conjectured in \cite{AP2021} when quantifying the convergence for some ad-hoc metric. A motivating and representative example is the one where $\phi(x)=x^2$ and in that case 
\[
\sqrt{n} \left( \mathbb{E}_X\left[S_n(X)^2\right]-1 \right)=\sqrt{n} \left( \frac{1}{n} \sum_{k=1}^n \frac{a_k^2+b_k^2}{2} -1 \right).
\]
In this particular example, the convergence thus reduces to the classical CLT for i.i.d. random variables and the Gaussian nature of the limit should not be surprising. 
Note that the $k-$th moments of the $\sin_c$ function appearing in the limit variance are well studied, and can be easily upper bounded. For example, following \cite{borwein2}, one has 
\[
\int_{\mathbb{R}}\left(\frac{\sin(x)}{x}\right)^2 dx={}{\pi}, \quad  \lim_{k \to+\infty} \sqrt{k} \int_{\mathbb{R}}\left|\frac{\sin(x)}{x}\right|^k dx=\sqrt{\frac{3\pi}{2}}.
\]
\end{rem}

\begin{rem}\label{rem.gauss}
The proof of Theorem \ref{Main-Gaussian} takes fully advantage of the Gaussian context, via the so-called Wiener chaos decomposition and the Malliavin--Stein approach. In particular, via Peccati--Tudor Theorem, one can restrict to a finite number of chaoses and the use of square field operator then allows to reduce the convergence in distribution to the convergence of a scalar sequence. 
\end{rem}

Let us now describe the second main result of the article, i.e. the fluctuations in the more general context of symmetric coefficients. 
We will work under the more restrictive condition on the test function  $\phi\in\mathbb{L}^2(\gamma)$
$$(\star)\quad \exists 0<\kappa<1, \;\phi(x)=\text{O}\left(e^{\kappa \frac{x^2}{2}}\right)\;\; \text{and} \;\; \exists A>0\;\;\text{s.t.}\;\;\sum_{k=0}^\infty |c_k(\phi)| k! A^k<\infty.$$

\begin{rem}\label{rem.star}
Note that condition $(\star)$ is in fact rather mild, in particular, it allows to consider test functions $\phi$ with arbitrary exponential growth rate. Indeed, we have the well known following Hermite decomposition of the exponential function, for $\alpha >0$
\[
e^{\alpha x} =  \sum_{k=0}^{\infty}c_k(\alpha) H_k(x), \quad \text{with} \quad  c_k(\alpha):=e^{\alpha^2/4}\frac{\alpha^k}{k!} ,
\]
so that, if $\alpha>0$ and for any $A>0$ such that $ A\alpha<1$, we have
\[
\sum_{k=0}^\infty |c_k(\alpha)| k! A^k = e^{\alpha^2/4} \sum_{k=0}^\infty \left( A \alpha\right)^k <+\infty.
\]
Since the above discussion can be extended to complex variables,  condition $(\star)$ in particular allows to consider characteristic functions.
Note moreover that the condition guarantees the finiteness of series of the following forms, which will naturally appear in the course of the proofs below, for $\beta, C, T$ arbitrary positive constants
\[
\sum_{k\geq 1} c_k(\phi)^2 k! k^{\beta}<+\infty, \quad \sum_{k\geq 1} \frac{c_k(\phi) k!}{\sqrt{(k-1)!}}<+\infty, \quad \sum_{k\geq 1} c_k(\phi) k! \frac{k^{\beta} C^k}{\lfloor k/T \rfloor !}<+\infty.
\]
\end{rem}
\noindent
The fluctuations for general symmetric coefficients are described by the next result. 
\begin{theo}\label{Main-pas-Gaussian}
Suppose that the variables $(a_k)$ and $(b_k)$ are all independent and identically distributed, their common law being symmetric, with unit variance and a finite moment of order six.  Suppose moreover that the test function $\phi$ satisfies condition ($\star$), then as $n$ goes to infinity, we have the convergence in distribution
\begin{equation}\label{Main-pas-Gaussian-eq}
\sqrt{n}\left(\mathbb{E}_X\left[\phi\left(S_n(X)\right)\right]-\gamma(\phi) \right)\xrightarrow[n\to\infty]{\text{Law}}~\mathcal{N}(0,\Sigma_\phi^2),
\end{equation}
where this time the variance of the limit Gaussian distribution is given by
\begin{equation}\label{variance-infinie-pas-Gaussien}
\Sigma_\phi^2:=\sigma_\phi^2+\frac{1}{{}{2}} c_2(\phi)^2 \left(\mathbb{E}_{\mathbb{P}}\left[a_1^4\right]-3\right).
\end{equation}
\end{theo}

\begin{rem}
Of course, the main novelty of this last result when compared to the first Theorem \ref{Main-Gaussian} is that the limit variance does not coincide with the one obtained in the Gaussian case. Therefore, the last Theorem \ref{Main-pas-Gaussian} is a typical example of a non-universality statement. Quite surprisingly, the correction term is simply given by the product of the kurtosis of the common distribution of the coefficients factorized by the second Hermite coefficient of the test function. As a result, if the latter vanishes -- it occurs e.g. if $\phi$ is an odd function-- then the limit variance is indeed universal. Going back to the simple example $\phi(x)=x^2$ mentioned in Remark \ref{rem.first} above,  the classical CLT applied to the i.i.d. variables $(a_k^2+b_k^2)/2 -1$ yields for example 
\[
\sqrt{n} \left( \mathbb{E}_X\left[S_n(X)^2\right]-1 \right)=\sqrt{n} \left( \frac{1}{n} \sum_{k=1}^n \frac{a_k^2+b_k^2}{2} -1 \right) \xrightarrow[n\to\infty]{\text{Law}}~\mathcal{N}\left(0,\frac{1}{2} \left( \mathbb E[a_1^4]-1\right)\right).
\]
Note that in the case of Rademacher coefficients, i.e. if $a_k,b_k$ take values in $\pm 1$, then one has $\sqrt{n} \left( \mathbb{E}_X\left[S_n(X)^2\right]-1 \right) \equiv 0$, so that the non-universality is then obvious.
\end{rem}

\begin{rem}
As already mentioned, this non-universality result can be thought as a first partial explanation to the surprising phenomenon that, when looking at the large degree asymptotics of the number of zeros of random trigonometric polynomials, the expected number of zeros is universal,  whereas the variance is not. Thanks to Kac formula, the number of zeros of the function $S_n$ can be expressed as an integral functional of $(S_n(X), S_n'(X))$ under $\mathbb P_X$, see \cite{nousAOP,AP2021}. The above Theorem \ref{Main-pas-Gaussian} shows that, dealing with more simple observables of the type $\mathbb E_X[\phi(S_n(X))]$, their fluctuations are already non-universal. The extension of Theorem \ref{Main-pas-Gaussian} to the bi variate framework of the function and its derivative is the object on ongoing work by the authors.
\end{rem}

The detailed proofs of Theorems  \ref{Main-Gaussian} and \ref{Main-pas-Gaussian} are the object of the next Sections \ref{sec.gauss} and \ref{sec.pasgauss} respectively. To facilitate the reading of the paper, the proofs of some of the technical estimates are postponed in the last Section \ref{sec.tech}. 

\subsection{Ingredients and heuristics of the proofs}\label{sec.heuristic}
Before giving the detailed proofs of our main results in the next sections, let us give here the heuristics of the proofs and describe the mathematical ingredients involved. We deliberately place ourselves in a slightly more general context since the method we develop here is not restricted to trigonometric sums, but would apply to more general sums of random variables.

\medskip

Suppose that we are given a sequence (or triangular array) of random variables $\{U_{i,n}\}_{1\le i \le n}$ defined on some probability space $(\Omega_1,\mathcal{F}_1,\mathbb{P}_1)$,  as well as another sequence of random variables $\{a_i\}_{i\ge 1}$ that are centered with variance and defined on another probability space $(\Omega_2,\mathcal{F}_2,\mathbb{P}_2)$. We also consider a test function $\phi:\mathbb{R}\to\mathbb{R}$ with enough integrability and which is centered with respect to the Gaussian distribution. We assume the following condition of normalization:

$$(C.N)~~\forall n\ge 1, \forall \omega_1 \in \Omega_1,\,\mathbb{E}_{\mathbb{P}_2}\left[\left(\sum_{i=1}^n U_{i,n}(\omega_1) a_i\right)^2\right]=\sum_{i=1}^n U_{i,n}(\omega_1)^2=1.$$

We are interested in proving a central limit Theorem for a quantity of the form $r_n~\mathbb{E}_{\mathbb{P}_1}\left[ \phi\left(\sum_{i=1}^n a_i U_{i,n}\right)\right]$, where $r_n$ is a suitable deterministic sequence. In the setting of this article, we have for instance $U_{i,n}=\cos(i X)/\sqrt{n}$, if we omit the sinus contribution. 

\medskip

\underline{Step 1:}~The first step is natural and consists of exploring this question in the somewhat simpler case of Gaussian coefficients, that is to say when the sequence $\{a_k\}_{k\ge 1}$ is a sequence of i.i.d. standard Gaussian variables. In a such case, the machinery of Wiener chaos is a perfectly adapted tool as one can expand the functional $\phi$ into  the orthogonal basis of Hermite polynomials:
\[
\phi(x)=\sum_{k=1}^\infty c_k(\phi) H_k(x),
\]
and one is left to establish a CLT for the Wiener expansion
\[
(W.E)\qquad r_n \,\mathbb{E}_{\mathbb{P}_1}\left[ \phi\left(\sum_{i=1}^n a_i U_{i,n}\right)\right]=\sum_{k=1}^\infty c_k(\phi) r_n\, \mathbb{E}_{\mathbb{P}_1}\left[ H_k\left(\sum_{i=1}^n a_i U_{i,n}\right)\right].
\]

Thanks to the normalization condition $(C.N)$, the quantity $r_n\, \mathbb{E}_{\mathbb{P}_1}\left[ H_k\left(\sum_{i=1}^n a_i U_{i,n}\right)\right]$ belongs to the Wiener chaos of order $k$ and the previous equality is the so called Wiener chaotic expansion of $r_n \, \mathbb{E}_{\mathbb{P}_1}\left[ \phi\left(\sum_{i=1}^n a_i U_{i,n}\right)\right]$ which theoretically exists under an $L^2$ condition. Indeed, the Wiener chaotic expansion simply consists in the orthonormal basis of polynomial for the Gaussian distribution on $\mathbb{R}^d$ if the dimension is finite or even in $\mathbb{R}^{\mathbb{N}}$ for an infinite dimensional setting. That being said, Wiener chaos display remarkable properties with respect to Gaussian convergence: (i) a sequence of random variables lying in a Wiener chaos of fixed order $k\ge1$ has a Gaussian limit in distribution if and only if its fourth cumulant tends to zero which is pretty easy to check in practice and (ii) the separate convergence in distribution can be upgraded to joint convergence. One can consult the dedicated articles \cite{peccati2005gaussian}  and \cite{nualart2005central} for a detailed overview.

\medskip
\underline{Step 2:}~Once the Gaussian case is resolved, one may still writes the expansion  (W.E) formally in the case where the $\{a_k\}_{k\ge 1}$ are not Gaussian anymore. The problem is that, one loses the rich structure of Wiener chaoses provided by the aforementioned properties (i) and (ii) and even the fact that the chaotic components are orthogonal. To overcome these issues, one needs to somehow ``simplify'' the expressions $r_n \,\mathbb{E}_{\mathbb{P}_1}\left[ H_k\left(\sum_{i=1}^n a_i U_{i,n}\right)\right]$ which are polynomial quantities of the sequence $(a_1,\cdots,a_n)$, in such way that
\[
(S)\qquad r_n\,  \mathbb{E}_{\mathbb{P}_1}\left[ {}{H_k}\left(\sum_{i=1}^n a_i U_{i,n}\right)\right]=Q_{k,n}(a_1,\cdots,a_n)+\mathcal{R}_{k,n},
\]
where $Q_{k,n}:\mathbb{R}^n\to\mathbb{R}$ is a {\bf homogeneous polynomial} of degree $k$ with no diagonal terms and $\mathcal{R}_{k,n}$ is a remainder which tends to zero as $n\to\infty$. This procedure is purely algebraic and as detailed in Section \ref{sec.combi} below, it can be carried out through the so called Newton--Girard formula which, given a set of numbers $\{x_1,\cdots,x_n\}$, relates in a polynomial way the symmetric expressions $e_{n,l}=\sum_{i_1<i_2<\cdots<i_l} x_{i_1}x_{i_2}\cdots x_{i_l}$ with the Newton sums $N_{n,l}=\sum_{i=1}^n x_i^l$. Concretely one gets
\[
e_{n,p}= (-1)^p \sum_{\substack{ m_1+2 m_2+\ldots+p m_p = p \\ m_1, \ldots, m_p \geq 0}} \prod_{j=1}^p\frac{\left (-N_{n,j}\right)^{m_j}}{m_j ! j^{m_j}}.
\]
Applying this to the case where $x_i=U_{i,n} a_i$ and assuming the natural condition that $\max_{1\le i\le n} |U_{i,n} a_i|=o(1)$ in some probabilistic sense, which roughly expresses the fact that no variable $U_{i,n}$ dominates, we get that $|N_{n,l}|\le \max_i|U_{i,n}| |a_i|^{l-2} \sum_{i} U_{i,n}^2 a_i^2$ which may be neglected for $l\ge 3$ in some probabilistic sense. Then, the Newton--Girard formula considerably simplifies at the first order as one gets
\[
e_{n,p}=(-1)^p \sum_{m_1+2m_2=p} \frac{(-N_{n,1})^{m_{1}} (-N_{n,2})^{m_2}}{m_1!m_2! 2^{m_2}}+\text{Remainder}.
\]
Noticing that, on the one hand $H_p(x)=\sum_{m_1+2m_2=p} (-1)^{m_1+m_2}\frac{x^m_{1} }{m_1!m_2! 2^{m_2}}$, and on the other hand that $N_{n,2}\approx 1$ which is ensured by the normalizing condition $(C.N)$ and the law of large numbers, one thus obtains the following simplified version of Equation $(S)$
\[
(S') \qquad r_n \, \mathbb{E}_{\mathbb{P}_1}\left[H_k\left(\sum_{i=1}^n U_{i,n} a_i\right)\right]\approx r_n\, k! \sum_{1\le i_1<i_2\cdots<i_k\le n} a_{i_1}\cdots a_{i_k}\mathbb{E}_{\mathbb{P}_1}\left[U_{i_1,n}\cdots U_{i_k,n}\right].
\]

\medskip
\underline{Step 3:}~Once these first two steps are achieved, the conclusion follows from the combination of two results borrowed to the theory of noise sensitivity. First of all, given an homogeneous sum of degree $k$, say 
\[
Q_k(x_1,\cdots,x_n)=\sum_{i_1<\cdots<i_k} f(i_1,\cdots,i_k) x_{i_1}\cdots x_{i_n},
\] 
one can associate its maximal influence function: 
\[
\tau:=\tau_n=\max_{1\le k \le n} \sum_{i_2<\cdots<i_d} f(k,i_2,\cdots,i_k)^2.
\] 
If the maximal influence $\tau$ is small (in a suitable quantified manner), then as established in \cite{mossel2010noise}, under suitable renormalization conditions, one can make the following asymptotic identification in law, where $g_1, \ldots, g_n$ are i.i.d Gaussian variables 
\[
Q_k(a_1,\cdots,a_n)\stackrel{\text{Law}}{\approx} Q_k(g_1,\cdots,g_n).
\] 
Keeping this in mind and coming back to our problematic, it is then sufficient to prove that all the maximal influence functions of the homogeneous sums 
\[
r_n\sum_{1\le i_1<i_2\cdots<i_k\le n} a_{i_1}\cdots a_{i_k}\mathbb{E}_{\mathbb{P}_1}\left[U_{i_1,n}\cdots U_{i_k,n}\right]
\] 
tend to zero as $n$ goes to infinity. It is indeed the case, using the observation that, as shown in \cite[Equation (1.10)]{nourdin2010invariance}, if an homogeneous sum evaluated in Gaussian standard random variables has a Gaussian limit (as establish in Step 1), then necessarily, its maximal influence function tends to zero. As a conclusion,
we can infer that the asymptotic behavior of the left hand side of Equation $(S')$ coincides with the one where we replace the non-Gaussian coefficients $\{a_k\}_{k\ge 1}$ by Gaussian ones.  The conclusion then follows from Step 1, where we established the desired CLT in the Gaussian framework. 

\section{Fluctuations in the Gaussian case} \label{sec.gauss}

The object of this section is to give the detailed proof of Theorem \ref{Main-Gaussian} describing the fluctuations in Salem--Zygmund CLT in the case of i.i.d. standard Gaussian coefficients. We will first treat the case where the base function $\phi$ is an Hermite polynomial and then deduce the general case where $\phi \in L^2(\gamma)$.

\subsection{Dealing with projections on Wiener chaoses} In this section, we shall first and specifically deal with the case where $\phi(x)=H_q(x)$. Note that in this latter case, 
we have $\int_{\mathbb{R}}H_q(x)d\gamma(x)=0$ for all $q \geq 1$ so that we shall in fact establish a central limit Theorem for the sequence of random variables $\left(\sqrt{n}~\mathbb{E}_X\left[H_q\left(S_n(X)\right)\right]\right)_{n \geq 1}.$ 
To simplify the expressions in the sequel, let us introduce the following notation, for all integers $n \geq 1$ and $q \geq 1$ 
\begin{equation}
F_{n,q} := \sqrt{n}~\mathbb{E}_X\left[H_q\left(S_n(X)\right)\right].
\end{equation}
Note that $F_{n,1}=\sqrt{n}~\mathbb{E}_X\left[H_1\left(S_n(X)\right)\right]=\sqrt{n}~\mathbb{E}_X\left[S_n(X)\right]=0$ so that the interesting cases naturally correspond to indices $q\ge 2$. The object of the section is to establish the following result.

\begin{theo}\label{Gaussian-hermite-clt}
Under the probability $\mathbb P$, for any $q\ge 2$, we have the following convergence in distribution as $n$ goes to infinity
\begin{equation}\label{Gaussian-hermite-clt-eq}
F_{n,q} \xrightarrow[n\to\infty]{\text{Law}}\mathcal{N}(0,\sigma_q^2)\quad \text{where}\quad {}{\sigma_q^2:=\frac{q!}{2\pi} \int_{\mathbb{R}}\left(\frac{\sin(x)}{x}\right)^q dx}.
\end{equation}
\end{theo}

\medskip
As already mentioned in the heuristic description given in Section  \ref{sec.heuristic},
the proof relies on the so-called \textit{Malliavin--Stein} approach. Let us explain briefly below the heart of this method in the finite dimensional setting. Note that, the quantities we manipulate here leave by essence in a finite dimensional Wiener space since for each $n\ge 1$, $S_n(X)$ involves only a finite number of Gaussian coefficients. Let us take $n\ge 1$ and let us place ourselves in the probability space $(\mathbb{R}^n,\mathcal{B}(\mathbb{R}^n),\gamma_n)$ with $\gamma_n$ the $n$-dimensional standard Gaussian distribution on $\mathbb{R}^n$. First, if $L:=\Delta-\vec{x}\cdot\nabla$ is the well known \textit{Ornstein--Ulhenbeck} operator that is symmetric with respect to $\gamma_n$, we have standard decomposition
\begin{eqnarray*}
L^2(\gamma_n)&=&\bigoplus_{k=0}^\infty \textbf{Ker}\left(L+k\textbf{I}\right),\\
\textbf{Ker}\left(L+k\textbf{I}\right)&=&\text{Vect}\left(\prod_{i=0}^n H_{k_i}(x_i)\Big{|}\sum_{i=0}^{n}k_i=n\right):=\hspace{-0.6 cm}\underbrace{\mathcal{W}_k.}_{k\text{-th Wiener chaos}}
\end{eqnarray*}
Then, introducing $\Gamma\left[\cdot,\cdot\right]=\nabla\cdot\nabla$ the square field or ``carr\'e du champ'' operator associated with $L$, for any $F\in\text{Ker}\left(L+k\text{I}\right)$ such that $\mathbb{E}\left[F^2\right]=1$ and for some constant $C_k$ only depending on $k$, the total variation distance between the variable $F$ and a standard Gaussian can be upper bounded by
\[
d_{TV}\left(F,\mathcal{N}(0,1)\right)\le C_k \sqrt{\text{Var}\left(\Gamma\left[F,F\right]\right)}.
\]

One might consult e.g. \cite{nourdin2009stein} or else Lemma \ref{lem.nol} below for a non quantitative analogous statement. In fact, in order to bound $\sqrt{\text{Var}\left(\Gamma\left[F,F\right]\right)}$ and establish the asymptotic Gaussianity, it is equivalent to show that $\E\left[F^4\right]-3\approx 0$, but we shall use instead the Poincar\'{e} inequality which ensures that $\text{Var}(Z)\le \mathbb{E}\left[\Gamma[Z,Z]\right]$ for any variable $Z$ in the domain $\mathbb{D}$ of $\Gamma\left[\cdot,\cdot\right]$. To facilitate the reading of this section, let us recall that operator $\Gamma\left[\cdot,\cdot\right]$ is bilinear, symmetric and obeys the chain rule.

\bigskip

In order to apply in our context the aforementioned Malliavin--Stein method, without any loss of generality we may consider that our coefficients $\{a_k,b_k\}_{1\le k \le n}$ are the coordinate mappings of the $2n$-dimensional probability space $\left(\mathbb{R}^{2n},\mathcal{B}\left(\mathbb{R}^{2n}\right),\gamma_{2n}\right)$. Next, we must observe that, for any $n\ge 1$,  for every fixed value of the random variable $X$ in $[0, 2\pi]$ and under $\mathbb{P}$, the random variable $S_n(X)$ is  a standard Gaussian random variable which belongs to the first Wiener chaos induced by the coefficients. Hence, $H_q\left(S_n(X)\right)$ and $\mathbb{E}_X\left[ H_q\left(S_n(X)\right)\right]$ belong to the $q$-th Wiener chaos as established in the next lemma.

\begin{lem}\label{lem.chaos}
For every integer $n\ge 1$, we have $\mathbb{E}_X\left[H_q\left(S_n(X)\right)\right]\in\mathcal{W}_q$.
\end{lem}

\begin{proof}[Proof of Lemma \ref{lem.chaos}]
As already stated above, for any fixed $x\in\mathbb{R}$ and under $\mathbb P$, we can observe that $S_n(x)=\frac{1}{\sqrt{n}}\sum_{k=1}^n a_k\cos(k x)+b_k\sin(k x)$ is a random variable with Gaussian distribution, which is centered and with variance
\begin{eqnarray*}
\mathbb{E}\left[S_n(x)^2\right]&=&\frac{1}{n}\sum_{k=1}^{n} \cos^2(kx)+\sin^2(kx)=1.
\end{eqnarray*}
Then, by definition of Wiener chaos, for every $q\ge 1$ it holds that $H_q(S_n(x))\in\mathcal{W}_q$. Let us recall that $\mathcal{W}_q$ is a closed subspace of $L^2(\mathbb{P})$ under the almost sure convergence. Thus, using for instance Riemann approximation we may write:

\begin{eqnarray*}
\mathbb{E}_X\left[H_q\left(S_n(X)\right)\right]&=\displaystyle{\frac{1}{2\pi}\int_0^{2\pi} H_q\left(S_n(x)\right)dx} =&\lim_{p\to\infty} \frac{{}{1}}{p}\sum_{k=0}^{p-1} H_q\left(S_n\left(\frac{2 k \pi}{p}\right)\right)\in \mathcal{W}_q.
\end{eqnarray*}
\end{proof}

\medskip

Given that, thanks to Lemma \ref{lem.chaos}, $(F_{n,q})_{n\ge 1}$ belongs to the $q$-th Wiener chaos, in order to prove that $F_{n,q}$ converges in distribution towards $\mathcal{N}(0,\sigma_q^2)$, it is necessary and sufficient to establish that $\Gamma\left[F_{n,q},F_{n,q}\right]$ converges towards the corresponding constant $q \sigma_q^2$, where $\Gamma$ is the aforementioned square field operator. The aforementioned convergence can be established either in probability or in $L^2(\mathbb{P})$ since these two topologies coincide on a finite sum of Wiener chaoses. More concretely we shall rely on the following lemma whose proof may be found for instance in \cite{nualart2005central}. We refer to \cite{nourdin2009stein} for stronger quantitative versions via Stein's method.

\begin{lem}\label{lem.nol}
Let $(X_n)_{n \geq 1}$ be a sequence of random variables on $(\Omega, \mathcal F, \mathbb P) $ such that that there exists $q \geq 1$ with $X_n\in\mathcal{W}_q$ for all $n \geq 1$ and such that, as $n$ goes infinity  $i)$ $\mathbb{E}\left[X_n^2\right]\to\sigma^2$ and $ii)$ $\text{Var}\left[\Gamma[X_n,X_n]\right]\to 0$. Then, under $\mathbb P$ and as $n$ goes to infinity, one has the following convergence in distribution $X_n\to \mathcal{N}(0,\sigma^2)$.
\end{lem}

In order to implement this strategy, one must compute $\Gamma[F_{n,q},F_{n,q}]$ and establishes that its variance tends to zero. To do so, we first establish the following lemma, where $D_n$ denotes the classical Dirichlet kernel 
\begin{equation}\label{eq.dirichlet}
D_n(x):=\frac{1}{n}\sum_{k=1}^n \cos( kx)=
\frac{1}{n}\cos\left(\frac{n+1}{2}x\right)\frac{\sin\left(\frac{nx}{2}\right)}{\sin\left(\frac{x}{2}\right)}.
\end{equation}
If $Y$ is and independent copy of $X$, also independent of the coefficients $\{a_k,b_k\}_{k\ge 1}$, the notation $\mathbb E_{X,Y}$ corresponds to the expectation with respect to $\mathbb P_X \otimes \mathbb P_Y$. In the same way, if $\vec{X}=(X_1,X_2)$ and $\vec{Y}=(Y_1,Y_2)$ are vectors whose entries are all independent copies of $X$ also independent of the coefficients $\{a_k,b_k\}_{k\ge 1}$, then $\mathbb E_{\vec{X},\vec{Y}}$ denotes the expectation with respect to the product measure $\mathbb P_{X_1} \otimes \mathbb P_{X_2}\otimes \mathbb P_{Y_1} \otimes \mathbb P_{Y_2}$. Finally and similarly, $\mathbb E_{\mathbb P, \vec{X},\vec{Y}}$ denotes the expectation with respect to the full product measure $\mathbb P \otimes \mathbb P_{X_1} \otimes \mathbb P_{X_2}\otimes \mathbb P_{Y_1} \otimes \mathbb P_{Y_2}$.

\begin{lem}\label{lem.control.var}
For all $q\geq 2$, there exists a constant $C_q>0$ only depending on $q$ such that for all $n \geq 1$
\[
\text{Var}\left[\Gamma[F_{n,q},F_{n,q}]\right]\le C_q n^2 \mathbb{E}_{\vec{X},\vec{Y}}\left[\left|D_n(X_1-Y_1)D_n(X_2-Y_2)D_n(X_1-X_2)\right|\right].
\]
\end{lem}

\begin{proof}[Proof of Lemma \ref{lem.control.var}]Let us first recall that $\mathbb E=\mathbb E_{\mathbb P}$ denotes the expectation with respect to the coefficients $\{a_k,b_k\}_{k\ge 1}$. Remark that for any $x,y \in \mathbb R$, we have 
\begin{equation}\label{eq.covar}
\mathbb E[S_n(x) S_n(y)] = \frac{1}{n} \sum_{k=1}^n \cos(kx)\cos(ky)+\sin(kx)\sin(ky) =D_n(x-y).
\end{equation}
Then, using both the chain rule and the bilinearity of the carr\'e du champ operator $\Gamma\left[\cdot,\cdot\right]$ we may write since $H_q'(x)=q H_{q-1}(x)$
\[
\begin{array}{ll}
\displaystyle{\Gamma\left[F_{n,q},F_{n,q}\right] } & =\displaystyle{ n \Gamma\left[ \mathbb E_X[H_q(S_n(X))],\mathbb E_Y[H_q(S_n(Y))]\right] } \\
& =\displaystyle{n q^2 \mathbb{E}_{X,Y}\left[H_{q-1}\left(S_n(X)\right)H_{q-1}\left(S_n(Y)\right)\mathbb E[S_n(X) S_n(Y)]\right]}\\
& =\displaystyle{n q^2 \mathbb{E}_{X,Y}\left[H_{q-1}\left(S_n(X)\right)H_{q-1}\left(S_n(Y)\right)D_n(X-Y)\right].}
\end{array}
\]
In order to control the variance of $\Gamma[F_{n,q},F_{n,q}]$ and avoid lengthy computations induced by the use of the multiplication formula on Wiener chaoses, we may rely on Poincar\'{e} inequality which infers that
$$\text{Var}\left[\Gamma\left[F_{n,q},F_{n,q}\right]\right]\le\mathbb{E}\left[\Gamma\left[\Gamma\left[F_{n,q},F_{n,q}\right],\Gamma\left[F_{n,q},F_{n,q}\right]\right]\right].$$

Keeping in mind that $\Gamma[\cdot,\cdot]$ is a diffusive operator, it obeys the chain rule and we may write for $(A,B,C,D)$ being for instance in Wiener chaoses -- but it remains true for variables in the domain $\mathbb{D}\cap \bigcap_{p\ge 1} L^p(\Omega_1,\mathcal{F}_1,\mathbb{P})$
$$\Gamma[A B, C D]=B D \Gamma[A,C]+B C \Gamma[A,D]+A C \Gamma[B,D]+AD\Gamma[B,C].$$
Then, applying the above chain rule to the variables $A=\mathbb{E}_{X_1}\left[H_{q-1}\left(S_n(X_1)\right)\right]$ and $B=\mathbb{E}_{Y_1}\left[H_{q-1}\left(S_n(Y_1)\right)\right]$ which belong to $\mathcal W_q$ and the variables $C=\mathbb{E}_{X_2}\left[H_{q-1}\left(S_n(X_2)\right)\right]$ and $D=\mathbb{E}_{Y_2}\left[H_{q-1}\left(S_n(Y_2)\right)\right]$ which belong to $\mathcal W_{q-1}$, and again using the bilinearity we get
\[\mathbb{E}\left[\Gamma\left[\Gamma\left[F_{n,q},F_{n,q}\right],\Gamma\left[F_{n,q},F_{n,q}\right]\right]\right]=n^2 q^2 (q-1)^2\times (\widetilde{A}+\widetilde{B}+\widetilde{C}+\widetilde{D}),
\]
with 
\[
\begin{array}{lll}
\widetilde{ A}&:=& \mathbb E_{\mathbb P, \vec{X},\vec{Y}}\left[D_n(X_1-Y_1)D_n(X_2-Y_2)D_n(X_1-X_2)\mathcal H_{n,q}(X_1,X_2,Y_1,Y_2)\right]\\
\widetilde{B}&:=&\mathbb E_{\mathbb P, \vec{X},\vec{Y}}\left[D_n(X_1-Y_1)D_n(X_2-Y_2)D_n(X_1-Y_2)\mathcal H_{n,q}(X_1,Y_2,Y_1,X_2)\right]\\
\widetilde{C}&:= &\mathbb E_{\mathbb P, \vec{X},\vec{Y}}\left[D_n(X_1-Y_1)D_n(X_2-Y_2)D_n(Y_1-X_2)\mathcal H_{n,q}(Y_1,X_2,X_1,Y_2)\right]\\
\widetilde{D}&:= &\mathbb E_{\mathbb P, \vec{X},\vec{Y}}\left[D_n(X_1-Y_1)D_n(X_2-Y_2)D_n(Y_1-Y_2)\mathcal H_{n,q}(Y_1,Y_2,X_1,X_2)\right],
\end{array}
\]
where we introduced the following notation to simplify the expressions 
\[
\mathcal H_{n,q}(X_1,X_2,Y_1,Y_2):=H_{q-2}\left(S_n(X_1)\right)H_{q-2}\left(S_n(X_2)\right) H_{q-1}\left(S_n(Y_1)\right)H_{q-1}\left(S_n(Y_2)\right).
\]
Noticing that each of the above variables is invariant by the operations $X_1\leftrightarrow Y_1$ or $X_2\leftrightarrow Y_2$, we get finally
\begin{equation}\label{eq.atilde}
\mathbb{E}\left[\Gamma\left[\Gamma\left[F_{n,q},F_{n,q}\right],\Gamma\left[F_{n,q},F_{n,q}\right]\right]\right] = 4  n^2 q^2 (q-1)^2 \widetilde{A}
\end{equation}
where, by the triangle inequality and Fubini inversion of integrals, we have
\[
0 \leq \widetilde{A} \leq \mathbb E_{\vec{X},\vec{Y}}\left[| D_n(X_1-Y_1)D_n(X_2-Y_2)D_n(X_1-X_2) | \mathbb E\left[| \mathcal H_{n,q}(X_1,X_2,Y_1,Y_2)|\right]\right].
\]
Then, using H\"{o}lder inequality, we can write
\[
\mathbb E\left[ | \mathcal H_{n,q}(X_1,X_2,Y_1,Y_2)|\right] \leq \prod_{i=1}^2 \mathbb{E}\left[H_{q-2}(S_n(X_i))^4\right]^{1/4}\mathbb{E}\left[H_{q-1}(S_n(Y_i))^4\right]^{1/4}.
\]
Here, one should again keep in mind that, for any given value of $x$, we have the identity in law $S_n(x)\sim \mathcal{N}(0,1)$ under $\mathbb{P}$, so that if $N\sim\mathcal N(0,1)$ under $\mathbb P$, the right hand side of the last inequality can be replaced by the following constant term.
\[
\mathbb E\left[ | \mathcal H_{n,q}(X_1,X_2,Y_1,Y_2)|\right] \leq \mathbb{E}\left[H_{q-2}(N)^4\right]^{\frac{1}{2}}\mathbb{E}\left[H_{q-1}(N)^4\right]^{\frac{1}{2}}.
\]
Injecting this estimates in Equation \eqref{eq.atilde} establishes the lemma with the constant
$$C_q:=4 q^2(q-1)^2 \mathbb{E}\left[H_{q-2}(N)^4\right]^{\frac{1}{2}}\mathbb{E}\left[H_{q-1}(N)^4\right]^{\frac{1}{2}}.$$
\end{proof}
Thanks to Lemma \ref{lem.control.var}, in order to prove that $\text{Var}\left[\Gamma\left[F_{n,q},F_{n,q}\right]\right]$ tends to zero as $n$ goes to infinity, we are thus left to prove that $n^2 \mathbb{E}_{\vec{X},\vec{Y}}\left[\left|D_n(X_1-Y_1)D_n(X_2-Y_2)D_n(X_1-X_2)\right|\right]$ tends to zero, which is the object of the next lemma.

\begin{lem}\label{lem.controldn}
Provided that $X_1,X_2,Y_1,Y_2$ are independent and uniformly distributed over the interval $[0,2\pi]$, we have
$$n^2 \mathbb{E}_{\vec{X},\vec{Y}}\left[\left|D_n(X_1-Y_1)D_n(X_2-Y_2)D_n(X_1-X_2)\right|\right]=\text{O}\left(\frac{\log(n)^3}{n}\right)\xrightarrow[n\to\infty]{} 0.$$
\end{lem}
\begin{proof}[Proof of Lemma \ref{lem.controldn}]
From the expression \eqref{eq.dirichlet} above of the Dirichlet kernel, for any $y\in[-n\pi,n\pi]$ we have
\begin{eqnarray*}
\left|D_n\left(\frac{y}{n}\right)\right|\le  \frac{|\sin(\frac{y}{2})|}{n |\sin(\frac{y}{2n})|}
\le \frac{\pi|\sin(\frac{y}{2})|}{|y|},\;\;\text{by using that}\;\;|\sin(\theta)|\ge \frac{2}{\pi}|\theta|\;\text{on}\;\left[0,\frac{\pi}{2}\right].
\end{eqnarray*}
Performing a change of variables and using the previous bound, we then obtain
\begin{eqnarray*}
&n^2& \mathbb{E}_{\vec{X},\vec{Y}}\left[\left|D_n(X_1-Y_1)D_n(X_2-Y_2)D_n(X_1-X_2)\right|\right]\\
&= &\frac{1}{{}{(2\pi)^4}n^2}\int_{[-n \pi, n\pi]^4}\frac{|\sin\left(\frac{x_1-y_1}{2}\right)|}{n \left|\sin\left(\frac{x_1-y_1}{2 n}\right)\right|}\frac{|\sin\left(\frac{x_2-y_2}{2}\right)|}{n \left|\sin\left(\frac{x_2-y_2}{2 n}\right)\right|}\frac{|\sin\left(\frac{x_1-x_2}{2}\right)|}{n \left|\sin\left(\frac{x_1-x_2}{2 n}\right)\right|}dx_1 dx_2dy_1dy_2\\
&\leq &{}{\frac{1}{16\pi n^2}}\int_{[-n \pi, n\pi]^4}\frac{|\sin\left(\frac{x_1-y_1}{2}\right)|}{\left|x_1-y_1\right|}\frac{|\sin\left(\frac{x_2-y_2}{2}\right)|}{\left|x_2-y_2\right|}\frac{|\sin\left(\frac{x_1-x_2}{2}\right)|}{\left|x_1-x_2\right|}dx_1 dx_2dy_1dy_2.
\end{eqnarray*}
Besides, for any value $x_1\in[-n\pi,n\pi]$ it holds that, for some absolute $C>0$
\begin{eqnarray*}
\int_{[-n \pi, n\pi]}\frac{|\sin\left(\frac{x_1-y_1}{2}\right)|}{\left|x_1-y_1\right|} dy_1&\le&\int_{[-2 n \pi, 2 n\pi]}\frac{|\sin\left(\frac{y_1}{2}\right)|}{\left|y_1\right|} dy_1\leq  C \log(n).
\end{eqnarray*}
As a result we deduce that
\begin{eqnarray*}
&&{}{\frac{1}{16\pi n^2}}\int_{[-n \pi, n\pi]^4}\frac{|\sin\left(\frac{x_1-y_1}{2}\right)|}{\left|x_1-y_1\right|}\frac{|\sin\left(\frac{x_2-y_2}{2}\right)|}{\left|x_2-y_2\right|}\frac{|\sin\left(\frac{x_1-x_2}{2}\right)|}{\left|x_1-x_2\right|}dx_1 dx_2dy_1dy_2\\
&=&{}{\frac{1}{16\pi n^2}}\int_{[-n \pi, n\pi]^3}\frac{|\sin\left(\frac{x_2-y_2}{2}\right)|}{\left|x_2-y_2\right|}\frac{|\sin\left(\frac{x_1-x_2}{2}\right)|}{\left|x_1-x_2\right|}\left(\int_{-n\pi}^{n\pi}\frac{|\sin\left(\frac{x_1-y_1}{2}\right)|}{\left|x_1-y_1\right|}dy_1\right) dx_1dx_2dy_2\\
&\le& {}{\frac{C\log(n)}{16\pi n^2}}\int_{[-n \pi, n\pi]^2}\frac{|\sin\left(\frac{x_1-x_2}{2}\right)|}{\left|x_1-x_2\right|}\left(\int_{-n\pi}^{n\pi}\frac{|\sin\left(\frac{x_2-y_2}{2}\right)|}{\left|x_2-y_2\right|}dy_2\right) dx_1dx_2\\
&\le& {}{\frac{C^2\log(n)^2}{16\pi n^2}}\int_{-n\pi}^{n\pi}\left(\int_{-n\pi}^{n\pi}\frac{|\sin\left(\frac{x_1-x_2}{2}\right)|}{\left|x_1-x_2\right|}dx_1\right)dx_2\le{}{\frac{C^3\log(n)^3}{8 n}}\to 0.
\end{eqnarray*}
\end{proof}
In view of Lemma \ref{lem.nol}, to complete the proof that $\Gamma[F_{n,q},F_{n,q}]$ tends to a constant $q \sigma_q^2$ it remains to show that the expectation converges. This is precisely the object of the next lemma.
\begin{lem}\label{variance-Fn}
For any $q\ge 2$, we have as $n$ goes to infinity
\[
\mathbb{E}_\mathbb{P}\left[\Gamma[F_{n,q},F_{n,q}]\right] \xrightarrow[n\to\infty]{} {}{\frac{q q!}{2\pi}} \int_{\mathbb{R}}\left( \frac{\sin(x)}{x} \right)^q(x)dx.
\]
\end{lem}

\begin{proof}[Proof of Lemma \ref{variance-Fn}]
Using Equation \eqref{eq.covar}, we write
\begin{eqnarray*}
\mathbb{E}\left[\Gamma[F_{n,q},F_{n,q}]\right]&=&q \mathbb{E}\left[F_{n,q}^2\right]=q \mathbb E \left[ \left( \sqrt{n} \mathbb{E}_{X}\left[ H_q(S_n(X))\right]\right)^2\right]\\
&=&q n \mathbb E \left[  \mathbb{E}_{X}\left[ H_q(S_n(X))\right] \mathbb{E}_{Y}\left[ H_q(S_n(Y))\right] \right]\\
&=&q n   \mathbb{E}_{X,Y}\left[ \mathbb E \left[ H_q(S_n(X)) H_q(S_n(Y))\right] \right]\\
&=&q q! n \mathbb{E}_{X,Y}\left[\mathbb E[S_n(X)S_n(Y)]^q\right]\\
&=&q q! n \mathbb{E}_{X,Y}\left[D_n(X-Y)^q\right]\\
&=&q q! n \mathbb{E}_{X}\left[D_n(X)^q\right]\\
&=&{}{\frac{q q!}{2\pi}} \int_{-n\pi}^{n\pi} \left(D_n\left( \frac{x}{n} \right)\right)^q dx.
\end{eqnarray*}
As $q\ge 2$, we observe that
$$\left| D_n\left( \frac{x}{n} \right)\right|=\left|\cos\left(\frac{(n+1)x}{2n}\right)\frac{\sin\left(\frac{x}{2}\right)}{n\sin\left(\frac{x}{2n}\right)}\right|\le \pi^q\frac{|\sin(x/2)|^q}{|x|^q}\in \mathbb{L}^1(\mathbb{R},dx).$$
Besides, for a fixed $x$, we also have as $n$ goes to infinity
$$D_n\left( \frac{x}{n} \right)=\cos\left(\frac{(n+1)x}{2n}\right)\frac{\sin\left(\frac{x}{2}\right)}{n\sin\left(\frac{x}{2n}\right)}\xrightarrow[n\to\infty]{} \frac{\sin(x)}{x}.$$
We conclude by Lebesgue dominated convergence.
\end{proof}
The combination of Lemmas \ref{lem.chaos} to Lemma \ref{variance-Fn} established above achieves the proof of Theorem \ref{Gaussian-hermite-clt} for the projection in each $q$-th Wiener chaos.

\begin{rem}\label{rem.jointe}
{}{Note that by the celebrated Peccati--Tudor Theorem allowing to deduce the joint convergence in a finite number of chaoses from the convergence of the projections, see e.g. \cite{peccati2005gaussian}, we may in fact infer that under the distribution $\mathbb P$ and for any fixed integers $q_1, \ldots, q_r$, we have
\[
(F_{n,q_1}, \ldots, F_{n,q_r}) \xrightarrow[n \to \infty]{Law}
\mathcal{N}\left(0,
\left(
\begin{array}{cccc}
\sigma_{q_1}^2 & 0&\cdots&0\\
0&\sigma_{q_2}^2& \ddots &0\\
\vdots&\ddots&\ddots& 0\\
0& \cdots &0 &\sigma_{q_r}^2
\end{array}
\right)
\right).
\]
}
\end{rem}

\subsection{The case of an infinite expansion on Wiener chaoses}

In this section we deal with the case of a general test function $\phi\in\mathbb{L}^2(\gamma)$ and complete the proof of Theorem \ref{Main-Gaussian}. As a first step one needs to establish the following lemma, where we recall that the Hermite coefficients $c_k(\phi)$ are such that
\[
\phi\stackrel{\mathbb{L}^2}{=}\sum_{k=0}^\infty c_k H_k, \quad \text{and} \quad F_{n,k} = \sqrt{n}~\mathbb{E}_X\left[H_k\left(S_n(X)\right)\right].
\]
\begin{lem}\label{lem.decomp-chaos}
We have the following identification in $\mathbb L^2$
\begin{equation}\label{decomp-chaos}
\sqrt{n} \left(\mathbb{E}_X\left[\phi(S_n(X))\right]-\int_{\mathbb{R}}\phi(x)d\gamma(x)\right)\stackrel{\mathbb{L}^2}{=}\sum_{k=2}^\infty c_k(\phi) F_{n,k}.
\end{equation}
\end{lem}
\begin{proof}[Proof of Lemma \ref{lem.decomp-chaos}]
First of all, relying on the proof of Lemma \ref{variance-Fn}, and using the fact that $|D_n(x)|\le 1$, we may write
\begin{eqnarray*}
\mathbb{E}_{\mathbb{P}}\left[F_{n,q}^2\right]&\le& q! n \mathbb{E}_{X}\left[\left|D_n(X)\right|^q\right]\leq   q! n \mathbb{E}_{X}\left[\left|D_n(X)\right|^2\right]\leq  C q!.
\end{eqnarray*}
Indeed, as seen in the proof of Lemma \ref{variance-Fn}, the sequence $n \, \mathbb{E}_{X}[\left|D_n(X)\right|^2]$ is convergent, hence bounded by some constant $C>0$. Gathering that (i) $F_{n,q}\in\mathcal{W}_q$, (ii) Wiener chaoses are orthogonal spaces and (iii) $\phi\in\mathbb{L}^2(\gamma)\Rightarrow \sum_{k=0}^\infty c_k(\phi)^2 k! <\infty$ together with the previous estimate implies that the series in Equation \eqref{decomp-chaos} is converging in $\mathbb{L}^2(\mathbb{P})$.
\end{proof}
Next, relying on Theorem \ref{Gaussian-hermite-clt} for the projection in each chaos and, as already noted in Remark \ref{rem.jointe}, the celebrated Peccati--Tudor Theorem allowing to deduce the joint convergence in a finite number of chaoses from the convergence of the projections, we may conclude that under $\mathbb P$ and for any fixed integer $Q\ge 1$ 
$$\sum_{k=2}^Q F_{n,q}\xrightarrow[n\to\infty]{\text{Law}}\mathcal{N}(0,\sum_{k=2}^Q k! \sigma_k^2).$$
This central convergence gathered with the fact that
$$\sup_{n\ge 1} \mathbb{E}_{\mathbb{P}}\left[\left(\sum_{k=Q+1}^\infty c_k(\phi)F_{n,k}\right)^2\right]\le C \sum_{k=Q+1}^\infty c_k(\phi)^2 k!=\underset{Q\to\infty}{o}\hspace{-0.3cm}\left(1\right),$$
achieves the proof of the Theorem \ref{Main-Gaussian}.

\section{Fluctuations for general symmetric coefficients} \label{sec.pasgauss}

We give in this section the detailed proof of Theorem \ref{Main-pas-Gaussian}. As already mentioned above, the methods employed in Section \ref{sec.gauss} are very specific to the Gaussian framework and do not apply for general symmetric distributions. We shall therefore replace the fruitful Malliavin--Stein approach by a combination of ideas among which
\begin{itemize}
\item a combinatorial formula, allowing to express power of sums of independent variables as homogeneous polynomials in the entries, up to negligible terms. This point is the object of the next Section \ref{sec.combi} ;
\item noise sensitivity methods to reduce the asymptotics behavior of homogeneous polynomials with independent entries to their analogue in the Gaussian framework. This point is the object of Section \ref{sec.noise} below.
\end{itemize}
Prior to this program, let us justify that, in the non-Gaussian framework, the variable $\mathbb{E}_X\left[\phi\left(S_n(X)\right)\right])$ is well defined if $\phi$ satisfies the sole condition $(\star)$, and that it is then well approximated by the associated Hermite sums.

\subsection{Hermite series approximation} \label{sec.well}

Under the condition $(\star)$, it might not completely clear how to give a meaning to the term $\mathbb{E}_X\left[\phi(S_n(X))\right]$. Indeed, one cannot use any more that $S_n(X)\sim\mathcal{N}(0,1)$ since $\{a_k,b_k\}_{k\ge 1}$ are not Gaussian any more. In order to bypass this problem, we may prove that, under the condition $(\star)$,  the Hermite expansion of $\phi$ converges uniformly on each compact set. Namely we have the following lemma.

\begin{lem}\label{decomp-gene-hermite}
If the test function $\phi$ satisfies the condition $(\star)$, then its Hermite expansion $\sum_{k=0}^ N c_k(\phi) H_k(x)$ converges uniformly on each compact set of $\mathbb{R}$. In particular, $\phi$ is continuous and almost surely with respect to $\mathbb{P}$ we have
\begin{equation}\label{decomp-gene-hermite-eq}
\sum_{k=0}^N c_k(\phi)\mathbb{E}_X\left[H_k(S_n(X))\right]\xrightarrow[N\to\infty]{} \mathbb{E}_X\left[\phi(S_n(X))\right].
\end{equation}
\end{lem}

\begin{proof}[Proof of Lemma \ref{decomp-gene-hermite}]
Let us recall the following criterion for Hermite series which may be found in \cite[p. 367]{sansone1959orthogonal}. For the sake of readability, we first state the criterion with the notations of the citation and we shall then translate it using our convention. In comparison with our probabilistic notations,  the \textit{physics} convention for Hermite polynomials is the following $h_n(x):=2^{\frac n 2} H_n(\sqrt{2} x)$. The initial criterion is the following. \par 
\smallskip
\noindent
\underline{Convergence criterion in physics notations.}  
Let $f$ be a measurable function which is integrable on any compact set of $\mathbb R$ and set let us define $F(x)=\int_0^x f(t) dt$ as well as $G(x)=-2 x f(x)+F(x)$. We assume
\[
i) \quad F(x)=\text{O}\left(e^{\kappa x^2}\right), \; \text{for some}\;  0<\kappa<1, \quad 
ii) \quad \int_{-\infty}^{\infty}e^{-x^2} G^2(x) dx<\infty.
\]
Then, setting 
\[
c_n:=\frac{1}{2^n n! \sqrt{\pi}}\int_{-\infty}^\infty e^{-x^2} F(x) h_n(x)dx,
\]
the series $\sum_{n=0}^\infty c_n h_n(x)$ converges uniformly on each compact set of $\mathbb R$ towards $F(x)$. 
\par 
\smallskip
\noindent
Let us now translate the above condition using our notations. We may write
\begin{eqnarray*}
F(x)&=&\sum_n c_n h_n(x) =\sum_n \left( \frac{1}{2^n n! \sqrt{\pi}}\int_{-\infty}^\infty e^{-u^2} F(x) 2^{n/2} H_n(\sqrt{2}u)du\right) 2^{n/2} H_n(\sqrt{2}x)\\
&=&\sum_n \frac{1}{n!} \left(\int_{-\infty}^\infty F\left(\frac{u}{\sqrt{2}}\right)H_n(u) \gamma(du)\right)H_n(\sqrt{2}x).
\end{eqnarray*}
Then, making the identification $\phi(x)=F(x/\sqrt{2})$, i.e. $F(x)=\phi(\sqrt{2} x)$, we have  naturally 
$F'(y)=\sqrt{2} \phi'(\sqrt{2} y)$ so that $G(x)=-2 x F'(y) +F(x) = -2 x \sqrt{2} \phi'(\sqrt{2} x) + \phi(\sqrt{2} x)$ and \par 
\[\begin{array}{ll}
\displaystyle{\int G^2(x) e^{-x^2} dx} & =  
\displaystyle{\int  \left(-2 x \sqrt{2} \phi'(\sqrt{2} x) + \phi(\sqrt{2} x)\right)^2 e^{-x^2} dx}\\
& \displaystyle{= \sqrt{\pi}\int \left(-2 x \phi'(x) + \phi(x)\right)^2 \gamma(dx)}.
\end{array}
\] 
Therefore, the above criterion translates as follows.
\par 
\smallskip
\noindent
\underline{Convergence criterion in probabilistic notations.} Let $\phi$ a real function such that 
\[
i) \quad \phi(x)=\text{O}\left(e^{\kappa \frac{x^2}{2}}\right), \; \text{for}\;  0<\kappa<1, \qquad 
ii) \;\;
\int_{-\infty}^{\infty}\left(\phi(x)-2{}{x} \phi'(x)\right)^2 \gamma(dx)<\infty.
\]
Then, the Hermite series $\sum_{k=0}^\infty c_k(\phi) H_k(x)$ converges uniformly on each compact set of the real line towards $\phi(x)$. 
\par 
\smallskip
\noindent
Naturally, the first point above coincides with the first point in condition $(\star)$. Otherwise, 
using that $H_k'(x)=k H_{k-1}(x)$ and $x H_{k-1}(x) = H_k(x)+(k-1)H_{k-1}(x)$, the Hermite expansion of $\phi(x)-2x \phi'(x)$ is given by
\[
\phi(x)-2x \phi'(x)=\sum_k \left[ c_{k}(\phi) -2 \left( k  c_{k}(\phi) +(k+2)k+1)c_{k+2}(\phi) \right)\right] H_k(x),
\]
and under the condition $(\star)$, following Remark \ref{rem.star}, we have indeed 
\[
\sum_k k! \left[ c_{k}(\phi) -2 \left( k  c_{k}(\phi) +(k+2)(k+1)c_{k+2}(\phi) \right)\right]^2<+\infty.
\]
As a consequence, condition $(\star)$ indeed implies the uniform convergence of the Hermite expansion on each compact set and therefore the continuity of $\phi$. 
It is now sufficient to notice that $S_n(x)$ is a trigonometric function of $x$ of degree $n$ and then is bounded. As a result, for any fixed integer $n$ and the coefficients $\{a_k,b_k\}_{k\ge 1}$ being fixed, $\sum_{k=0}^N c_k(\phi) H_k (S_n(x))$ converges uniformly on $[-\pi,\pi]$ towards $\phi(S_n(x))$ as $N$ goes to infinity. Integrating with respect to the Lebesgue measure on $[-\pi,\pi]$ gives the desired conclusion.
\end{proof}
\subsection{Newton--Girard formula: the link with Hermite polynomials}\label{sec.combi}
This section contains the combinatoric material which is necessary to establish the main Theorem \ref{Main-pas-Gaussian}. More precisely, we shall relate Hermite polynomials evaluated at  $S_n(X)$ with some homogeneous polynomials without diagonal terms in the coefficients $\{a_k,b_k\}_{k\geq 1}$, the latter being easier to manipulate in view of establishing a central limit Theorem. We first recall below the explicit expression of Hermite polynomials
\[
\forall p\ge 0,\quad \, H_p(x) = p ! \sum_{m=0}^{\lfloor p/2\rfloor} \frac{(-1)^m x^{p-2m}}{(p-2m) ! m ! 2^{m}}.
\]
Then, we fix the following notations for the so-called \textit{Newton sums}
\[
\forall p\ge 0,\,\quad  N_{n,p}(X):=\frac{1}{n^{p/2}}\sum_{k=1}^n (a_k \cos(kX)+b_k \sin(kX))^p,
\]
\[
\forall p \ge 2,\quad \, e_{n,p}(X) := \frac{1}{n^{p/2}} \sum_{\substack{k_1 \ldots, k_p \in [|1, n|] \\ k_1 < \ldots < k_p}} \prod_{i=1}^p \left( a_{k_i} \cos(k_i X)+b_{k_i} \sin(k_i X)\right),
\]
and we adopt the conventions $e_{n,1}(X)=N_{n,1}(X)$ and $e_{n,0}(X)=1$. In particular, with the notations of the introduction, we have naturally $N_{n,1}(X)=S_n(X)$. The so-called Newton--Girard formula relates symmetric polynomials with the Newton sums, see e.g. Equation (2.14') in \cite{macdo}. Namely,  for all $p\ge 1$, we have

\begin{equation}\label{Newton-Girard}
e_{n,p}(X) = (-1)^p \sum_{\substack{ m_1+2 m_2+\ldots+p m_p = p \\ m_1, \ldots, m_p \geq 0}} \prod_{j=1}^p\frac{\left (-N_{n,j}(X)\right)^{m_j}}{m_j ! j^{m_j}}.
\end{equation}
A key step in our approach is to decompose the aforementioned formula according to the positivity of $m_j$ for $j\geq 3$. Namely, one may write
\[
e_{n,p}(X) =M_{n,p}(X) + R_{n,p}(X),
\]
with the convention that $R_{n,0}(X)=R_{n,1}(X)=R_{n,2}(X)=0$ and 
\begin{eqnarray}
\forall p\ge 0,\quad M_{n,p}(X) & := \displaystyle{(-1)^p \sum_{\substack{ m_1+2 m_2 = p \\ m_1, m_2 \geq 0}} \prod_{j=1}^p\frac{\left (-N_{n,j}(X)\right)^{m_j}}{m_j ! j^{m_j}}},\\
\forall p\ge 3,\quad R_{n,p}(X) & := \displaystyle{(-1)^p \sum_{\substack{  m_1+2 m_2+\ldots+p m_p = p \\ m_1, \ldots, m_p \geq 0 \\
\exists j \geq 3, \, m_j>0}} \prod_{j=1}^p\frac{\left (-N_{n,j}(X)\right)^{m_j}}{m_j ! j^{m_j}}} \label{eq.Rn}.
\end{eqnarray}
Note that $M_{n,p}(X)$ can itself be rewritten as
\[
\begin{array}{ll}
M_{n,p}(X) & = \displaystyle{(-1)^p \sum_{m=0}^{\lfloor p/2\rfloor} \frac{\left (-N_{n,1}(X)\right)^{p-2m} \left (-N_{n,2}(X)\right)^{m}  }{(p-2m) ! m ! 2^{m}}} \\
\\
& = \displaystyle{ \sum_{m=0}^{\lfloor p/2\rfloor} \frac{(-1)^m \left(N_{n,1}(X)\right)^{p-2m} \left [\left( N_{n,2}(X)-1+1\right)^{m} \right] }{(p-2m) ! m ! 2^{m}}} \\
\\
& = \displaystyle{ \sum_{m=0}^{\lfloor p/2\rfloor} \frac{(-1)^m \left(N_{n,1}(X)\right)^{p-2m} }{(p-2m) ! m ! 2^{m}}\left [\sum_{k=0}^m \binom{m}{k}\left(N_{n,2}(X)-1\right)^k \right] } \\
\\
& = \displaystyle{ \sum_{k=0}^{\lfloor p/2\rfloor}\sum_{m=k}^{\lfloor p/2\rfloor} \frac{(-1)^m \left(N_{n,1}(X)\right)^{p-2m} }{(p-2m) ! k ! (m-k)! 2^{m}} \left(N_{n,2}(X)-1\right)^k } \\
\\
& = \displaystyle{ \sum_{k=0}^{\lfloor p/2\rfloor} \frac{(-1)^k}{k!2^k} \left(N_{n,2}(X)-1\right)^k\sum_{m=k}^{\lfloor p/2\rfloor} \frac{(-1)^{m-k} \left(N_{n,1}(X)\right)^{p-2k-2(m-k)} }{(p-2k-2(m-k)) ! (m-k)! 2^{m-k}} } \\
\\
& \stackrel{m\to m-k}{=} \displaystyle{ \sum_{k=0}^{\lfloor p/2\rfloor} \frac{(-1)^k}{k!2^k} \left(N_{n,2}(X)-1\right)^k\sum_{m=0}^{\lfloor (p-2k)/2\rfloor} \frac{(-1)^{m} \left(N_{n,1}(X)\right)^{p-2k-2m} }{(p-2k-2m) ! m! 2^{m}} } \\
\\
& = \displaystyle{ \sum_{k=0}^{\lfloor p/2\rfloor} \frac{(-1)^k}{k!2^k} \left(N_{n,2}(X)-1\right)^k \frac{H_{p-2k}(N_{n,1}(X))}{(p-2k)!} } .
\end{array}
\]
In other words, we have for every $p\ge 0$,
\begin{equation}\label{eq:premagic}
 \frac{H_{p}(N_{n,1}(X))}{p!} =e_{n,p}(X) - \underbrace{\sum_{k=1}^{\lfloor p/2\rfloor} \frac{\left(1-N_{n,2}(X)\right)^k}{k!2^k}  \frac{H_{p-2k}(N_{n,1}(X))}{(p-2k)!}}_{=0\,\text{if}\,p=0}-R_{n,p}(X).
\end{equation}
\noindent
Then, interchanging the role played by $H_{k}(N_{n,1}(X))$ and $e_{n,k}(X)-R_{n,k}(X)$ in the last Equation \eqref{eq:premagic}, leads to the following key formula.

\begin{prop}\label{Magic-formula}
For every $p\ge 1$ we have the relation
\begin{equation}\label{Magic-eq}
\frac{H_{p}\left(N_{n,1}(X)\right)}{p!}=\sum_{k=0}^{\lfloor \frac p 2\rfloor}\frac{(-1)^k (1-N_{n,2}(X))^k}{2^k k!}\left(e_{n,p-2k}(X)-R_{n,p-2k}(X)\right).
\end{equation}
\end{prop}

\begin{proof}[Proof of Proposition \ref{Magic-formula}]
Let us first assume that $p$ is an even integer, namely $p=2q$ with $q\ge 2$. If we formally introduce the notations $t=\frac{1-N_{n,2}(X)}{2}$ and for all $k \geq 0$
\[
a_k:=\frac{H_{2k}\left(N_{n,1}(X)\right)}{k!}, \quad b_k:=e_{n,2k}(X)-R_{n,2k}(X), \quad 
c_k(t):=\frac{t^k}{k!},
\]
then Equation \eqref{eq:premagic} reads as follows, for all $q \geq 1$
\begin{eqnarray*}
a_q=b_q-\sum_{k=1}^q \frac{t^k}{k!} a_{q-k},
\end{eqnarray*}
which in turn implies that for all $q\ge 1$
\[
a_q+\sum_{k=1}^{q}\frac{t^k}{k!} a_{q-k}= \sum_{k=0}^q\frac{t^k}{k!} a_{q-k}=\left[a\ast c(t)\right]_q=b_q.
\]
Since our conventions entail that $b_0=a_0 c_0=1$ we have $a\ast c(t)=b$ which can be reversed so that $a=b\ast c(-t)$. As a result, for all $q\ge 1$ we get indeed 
\[
\frac{H_{2q}\left(N_{n,1}(X)\right)}{(2q)!}=\sum_{k=0}^q \frac{(-1)^k (1-N_{n,2}(X))^k}{2^k k!}\left(e_{n,(2q-2k)}(X)-R_{n,(2q-2k)}(X)\right).
\]
The case of odd integers can be managed exactly in the same way. Precisely, if we set 
\[ 
\tilde{a}_k:=\frac{H_{2k+1}\left(N_{n,1}(X)\right)}{k!}, \quad \tilde{b}_k=e_{n,2k+1}(X)-R_{n,2k+1}(X),
\]
then Equation \eqref{eq:premagic} becomes this time, for every $q\ge 0$
\begin{eqnarray*}
\tilde{a}_q=\tilde{b}_q-\sum_{k=1}^q \frac{t^k}{k!} \tilde{a}_{q-k},
\end{eqnarray*}
which in turn implies for all $q\ge 0$ that $\tilde{a} \ast c(t)=\tilde{b}(t)$. Reversing again the relations we naturally obtain for every $q\ge 0$ that
\[
\frac{H_{2q+1}\left(N_{n,1}(X)\right)}{(2q+1)!}=\sum_{k=0}^q \frac{(-1)^k (1-N_{n,2}(X))^k}{2^k k!}\left(e_{n,2(q-k)+1}(X)-R_{n,2(q-k)+1}(X)\right).
\]
\end{proof}
\noindent
Let us introduce the following quantity, which will play an important role in the sequel
\begin{equation}\label{eq.rnp}
\mathcal{R}_{n,p}:=\mathbb{E}_X\left[-R_{n,p}(X)+\sum_{k=1}^{\lfloor \frac p 2\rfloor}\frac{(-1)^k (1-N_{n,2}(X))^k}{k!}\left(e_{n,p-2k}(X)-R_{n,p-2k}(X)\right)\right].
\end{equation}
Then, taking the expectation with respect to $\mathbb E_X$ in Equation \eqref{Magic-eq} yields the key decomposition 
\begin{equation}\label{eq.magic2}
\frac{1}{p!} \times \mathbb E_X \left[ H_p(S_n(X)) \right] = \mathbb E_X \left[ e_{n,p}(X) \right] + \mathcal{R}_{n,p}.
\end{equation}
As we shall see in the next section, the term $\mathcal{R}_{n,p}$ can be thought as a remainder term as $n$ goes to infinity, so that Equation \eqref{eq.magic2} in fact allows to reduce the study of the asymptotics of $\sqrt{n} \, \mathbb E_X \left[ H_p(S_n(X)) \right] $ to the one of $\sqrt{n} \, \mathbb E_X \left[ e_{n,p}(X) \right]$. The advantage and the pertinence of this approach is that there exists universality results for homogeneous polynomials that allow to establish central limit Theorems outside the Gaussian context. In some sense, the last Equation \eqref{eq.magic2} will enable us to extend Gaussian techniques used in the proof of Theorem \ref{Main-Gaussian} to the non-Gaussian framework.

\subsection{Some technical estimates}\label{sec.esti}

As announced just above, the goal of this section is to establish that, as $n$ goes to infinity, the term $\mathcal{R}_{n,p}$ in the last Equation \eqref{eq.magic2} goes to zero in a suitable sense, keeping track of the dependence in the parameter $p$. To facilitate the reading of the paper, the (rather technical) proofs of all the results stated in this section are postponed in the last Section \ref{sec.tech} of the article. \par
\medskip
The first  lemma below deals with an $L^2$ estimate with respect to $\mathbb{P}$ of terms of the form $\mathbb{E}_X\left[e_{n,p}(X)(1-N_{n,2}(X))\right]$ which appear in the definition of $\mathcal R_{n,p}$ in Equation \eqref{eq.rnp}. Namely, we have the following upper bound.

\begin{lem}\label{lem.reste0}
There exists a positive constant $C$ such that, as $n$ goes to infinity,
\[
n \, \mathbb E\left[ \mathbb E_X\left[ e_{n,p}(X) (1-N_{n,2}(X)) \right]^2\right] \leq\frac{C}{n (p-1)!}.
\]
\end{lem}

The next lemma allows to treat, in a $\mathbb P-$almost sure sense, the analogue quantities with higher powers of $(1-N_{n,2}(X))$. Again, we search for estimate which behaves nicely in function of both variables $n$ and $p$.

\begin{lem}\label{lem.reste1}
There exists a positive constant $C(\omega)$ such that for every $j\ge 2$ and every $p,n\ge 1$ we have $\mathbb{P}-$almost surely
\[
\sqrt{n}\left|\mathbb{E}_X\left[\left(1-N_{n,2}(X)\right)^j e_{n,p}(X)\right]\right|\le C(\omega)^{j}\sqrt{\log(n)}^j \frac{n^{\frac{7}{8}-\frac{j}{2}}}{p!^{\frac 1 4}}.
\]
\end{lem}

Finally, the third lemma below provides an almost sure estimation of the quantity $|R_{n,p}(X)|$ for which it appears crucial to track the dependency with respect to the parameter $p$. In the statement below, we separate small values of $p$ from large ones since the arguments in the proof differ. Note that the hypothesis that the sixth moment of the random coefficient is finite plays an important role here.

\begin{lem}\label{bound.rnp}
There exists a positive constant $C(\omega)$ such that for $p\in\{3,4,5,6\}$ and uniformly in $X$ we have $\mathbb{P}-$almost surely
\begin{equation}\label{eq-Rn1}
\left|R_{n,p}(X)\right|\leq \frac{C(\omega){}{\log(n)^2}}{n}, 
\end{equation}
and for all $p\ge 7$
\begin{equation}\label{eq-Rn2}
\left|R_{n,p}(X)\right|\le C(\omega)~{}{\log(n)^{p/2}} \max\left( \frac{1}{n} \frac{p^7}{\left[\frac{p}{72}\right]!},\frac{1}{n^{\frac{3p}{20}}}\right).
\end{equation}
\end{lem}

Based on the previous estimates, we can now formally state in which sense the family of variables $(\mathcal R_{n,p})$ are negligible as $n\to\infty$. Precisely, we have the following theorem, whose detailed proof is also given in the last Section \ref{sec.tech} of the article.

\begin{theo}\label{Remainder-main}
Under the condition $(\star)$, the following sequence of random variables is well defined and tends to zero in distribution under $\mathbb P$ as $n$ goes to infinity
\begin{equation}\label{Remainder-main-eq}
\sqrt{n}\sum_{p=3}^\infty p! \, c_p(\phi) | \mathcal R_{n,p} |\xrightarrow[n\to\infty]{\text{Law}}~0.
\end{equation}
\end{theo}

%
%
%
%
%
%
Let us now go back to the key Equation \eqref{eq.magic2} and examine how the above estimates allow to reduce the convergence of $\sqrt{n}\, \mathbb E_X[ \phi(S_n(X)]$ to the one of the homogeneous sums $(\sqrt{n} \, \mathbb E_X[ e_{n,p}(X)])_{p \geq 1}$.
Since, $H_0 \equiv 1$ and $H_1(x)=x$, we have 
\[
c_0(\phi)=\int_{\mathbb{R}}\phi(x)\gamma(dx), \quad \mathbb{E}_X\left[H_1\left(S_n(X)\right)\right]=\mathbb E_X[S_n(X)]=0,
\]
and thus, relying on Lemma \ref{decomp-gene-hermite}, we may infer that almost surely with respect to $\mathbb{P}$ we have
\[
\sqrt{n}\left(\mathbb{E}_X\left[\phi\left(S_n(X)\right)\right]-\int_{\mathbb{R}}\phi(x)\gamma(dx)\right)
=
\sqrt{n}\lim_{q\to\infty} \sum_{p=2}^q c_p(\phi)\mathbb{E}_X\left[H_p\left(S_n(X)\right)\right].
\]
Then, thanks to Proposition \ref{Magic-formula} (Newton--Girard decomposition), we can write for every $q\ge 3$ that
\[
\begin{array}{l}
\displaystyle{\sqrt{n}\sum_{p=1}^q c_p(\phi)\mathbb{E}_X\left[H_p\left(S_n(X)\right)\right]=c_2(\phi)~\sqrt{n}~\mathbb{E}_X\left[H_2\left(S_n(X)\right)\right]}\\
\\
\displaystyle{+\sqrt{n}\sum_{p=3}^q~p! c_p(\phi)\mathbb{E}_X\left[e_{n,p}(X)\right]
+\sqrt{n}~\sum_{p=3}^q p!c_p(\phi) \mathcal{R}_{n,p}}.
\end{array}
\]
Next, we use the last Theorem \ref{Remainder-main} to handle the remainder and we deduce that
\[
\lim_{q\to\infty}\sqrt{n}~\sum_{p=3}^q p!c_p(\phi) \mathcal{R}_{n,p}\xrightarrow[n\to\infty]{\text{Law}}~0.
\]
As a result, to study the limit behavior of $\sqrt{n}\, \mathbb E_X[ \phi(S_n(X)]$ as $n$ goes to infinity, 
one is then left to studying the limit in distribution of the infinite series 
\[
\sqrt{n}~c_2(\phi)\mathbb{E}_X\left[H_2\left(S_n(X)\right)\right]+\sqrt{n}\sum_{p=3}^\infty~p! c_p(\phi)\mathbb{E}_X\left[e_{n,p}(X)\right],
\]
where the last sum converges $\mathbb{P}-$almost surely. Moreover, one can even restrict to the case of a finite sum. Indeed, let us recall that 
\[
\sqrt{n}\, \mathbb{E}_X\left[e_{n,p}(X)\right] = \frac{1}{n^{p/2}} \sum_{1 \leq k_1 < \ldots <k_p\leq n} \mathbb E_X\left[  \prod_{i=1}^p\left(a_{k_i} \cos(k_i X)+b_{k_i} \sin(k_i X)\right)\right].
\]
A simple computation then yields that 
\[
\begin{array}{ll}
\displaystyle{n p!^2 \mathbb{E}\left[\mathbb{E}_X\left[e_{n,p}(X)\right]^2\right]} & = \displaystyle{n p!^2 \mathbb{E}_{\mathbb{P}, X, Y}\left[e_{n,p}(X) e_{n,p}(Y)\right]}\\
& = \displaystyle{n p!^2 \frac{1}{n^p}  \sum_{k_1<k_2<\cdots<k_p} \mathbb{E}_{X,Y}\left[\prod_{i=1}^p \cos \left( k_i\left(X-Y\right)\right)\right]}\\
& \displaystyle{=n p!^2 \frac{1}{n^p}  \sum_{k_1<k_2<\cdots<k_p} \mathbb{E}_{X}\left[\prod_{i=1}^p \cos \left( k_i X\right)\right]}\\
& \displaystyle{\le n p!^2 \frac{1}{n^p}\sum_{k_1<k_2<\cdots<k_{p-1}} 1\le n p!^2 \frac{1}{n^p} \frac{n^{p-1}}{(p-1)!}=p p!.}
\end{array}
\]
As a result, keeping in mind that for every $n\ge 1$ the random variables $\{\mathbb{E}_{{}{X}}\left[ e_{n,p}(X)\right]\}_{p\ge 1}$ form an orthogonal system we may infer that for any integer $Q \geq 1$
\[
\begin{array}{l}
\displaystyle{\mathbb{E}\left[\left(\sqrt{n}\sum_{p=Q}^\infty p! c_p(\phi) \mathbb{E}_{{}{X}}\left[e_{n,p}(X)\right]\right)^2\right] {}{=} \sum_{p=Q}^\infty n c_p^2(\phi) p!^2 {}{\mathbb{E}\left[\mathbb{E}_X\left[e_{n,p}(X)\right]^2\right]}}
\displaystyle{\le  \sum_{p=Q}^\infty c_p^2(\phi) p p!.}
\end{array}
\]
Thanks to condition $(\star)$, as noted in Remark \ref{rem.star}, the rest of the last series goes to zero as $Q$ goes to infinity. This leads in particular to
\begin{equation}\label{Juste-nbre-fini-de-termes}
\limsup_{Q\to\infty}\sup_{n\ge 1} \mathbb{E}\left[\left(\sqrt{n}\sum_{p=Q}^\infty p! c_p(\phi) \mathbb{E}_{{}{X}}\left[e_{n,p}(X)\right]\right)^2\right]=0.
\end{equation}
As a conclusion, to study the limit behavior of $\sqrt{n}\, \mathbb E_X[ \phi(S_n(X)]$ as $n$ goes to infinity, it is then sufficient to compute the limit in distribution of the finite sum
\begin{equation}\label{eq.key22}
\sqrt{n}~c_2(\phi)\mathbb{E}_X\left[H_2\left(S_n(X)\right)\right]+\sqrt{n}\sum_{p=3}^Q~p! c_p(\phi)\mathbb{E}_X\left[e_{n,p}(X)\right].
\end{equation}
The situation becomes now considerably simpler as we deal now with homogeneous, finite degree polynomial expressions in terms of the coefficients $\{a_k,b_k\}_{k\ge 1}$.

\subsection{Universality principles for homogeneous sums} \label{sec.noise}
We can now complete the proof of our second main Theorem \ref{Main-pas-Gaussian}.
The motivation behind the introduction of the combinatorics material in Section \ref{sec.combi} above, yielding to the study of the asymptotics of the last Equation \eqref{eq.key22}, is precisely that remarkable central limit Theorems and invariance principles for homogeneous polynomial expressions are already available in the literature, sometimes referred as noise stability or noise sensitivity results, see \cite{mossel2010noise,nourdin2010invariance} and the references therein.
\par
\medskip
The key observation here is that all the  variables $\sqrt{n}\left(\mathbb{E}_X\left[e_{n,p}(X)\right]\right)_{3\le p \le Q}$ are then homogeneous polynomials evaluated in the variables $\{a_k,b_k\}_{1\le k \le n}$, without diagonal terms in the precise sense of \cite[Definition 1.1]{nourdin2010invariance}. Namely, in order to stick to the notations of the latter reference, let us first remark that the whole family of random variables $\left(\sqrt{n}\, \mathbb{E}_X\left[e_{n,p}(X)\right]\right)_{n \geq 1,  p \geq 3}$ has the same distribution under $\mathbb P$ as the new family $\left(\sqrt{n}\, \mathbb{E}_X\left[\widetilde{e}_{n,p}(X)\right]\right)_{n \geq 1,  p \geq 3}$ where we have set
\[
\widetilde{e}_{n,p}(X):= \frac{1}{n^{p/2}} \sum_{1 \leq k_1 < \ldots <k_p\leq n} \prod_{i=1}^p\left(a_{2k_i} \cos(k_i X)+a_{2k_i-1} \sin(k_i X)\right).
\]
If we introduce the set 
\[
B_{2n}^p=\left \lbrace  (k_1, \ldots, k_p), \; 1\leq k_i \leq2 n,\; \begin{array}{l} \nexists i\neq j, \, k_i=k_j \\
\nexists (i,j), \; k_i \in 2\mathbb Z, \; k_j=k_i-1
\end{array}\right \rbrace,
\]
 then expanding the product, we can write alternatively
\[
\sqrt{n}\, \mathbb{E}_X\left[\widetilde{e}_{n,p}(X)\right]=\sum_{1 \leq k_1 < \ldots < k_p\leq {}{2n}} f_{2n,p}(k_1, \ldots, k_p) a_{k_1} \ldots a_{k_p},
\]
with $f_{n,p}$ being the following symmetric function vanishing on the diagonal
\[
f_{2n,p} (k_1, \ldots, k_p):= \frac{1}{n^{\frac{p-1}{2}}}  \,\mathbb E_X \left[ \prod_{\substack{i, \, k_i \in 2\mathbb Z \\ j, \, k_j \in 2\mathbb Z-1}} \cos\left( \frac{k_i}{2}X \right)  \sin\left( \frac{k_j+1}{2}X \right)\right] \mathds{1}_{(k_1, \ldots, k_p) \in B_n^p}.
\]
In order to prove a central limit Theorem for both sequences 
\[
\left(\sqrt{n}\, \mathbb{E}_X\left[\widetilde{e}_{n,p}(X)\right]\right)_{n \geq 1,  3 \leq p \leq Q} \quad \text{and} \;\; \left(\sqrt{n}\, \mathbb{E}_X\left[e_{n,p}(X)\right]\right)_{n \geq 1,  3 \leq p \leq Q}
\]
we will apply Theorem 1.2 of \cite{nourdin2010invariance}, which allows to reduce the general central limit Theorem to its Gaussian analogue.  \par
\medskip
\noindent
\underline{The Gaussian case}: suppose first that the variables $\{a_k,b_k\}_{k\ge 1}$ are independent standard Gaussian random variables.  Then, relying on Theorem \ref{Gaussian-hermite-clt} and Remark \ref{rem.jointe} above, we may infer that 
\[
\left( \sqrt{n}\, \mathbb{E}_X\left[H_p(S_n(X))\right]\right)_{3 \leq p\leq Q} \xrightarrow[n\to\infty]{\text{Law}}~\mathcal{N}\left(0,
\left(
\begin{array}{cccc}
\sigma_{3}^2 & 0&\cdots&0\\
0&\sigma_{4}^2& \ddots &0\\
\vdots&\ddots&\ddots& 0\\
0& \cdots &0 &\sigma_{Q}^2
\end{array}
\right)
\right).
\]
Now, for $p\in\{3,\cdots,Q\}$, applying Equation \eqref{eq.magic2}, we have the decomposition
\[
\sqrt{n}\, \mathbb{E}_X\left[H_p(X)\right]=\sqrt{n}\, p! \, \mathbb{E}_X\left[e_{n,p}(X)\right]+\sqrt{n}\, p! \,\mathcal{R}_{n,p}.
\]
One may then apply Theorem \ref{Remainder-main} to the particular case of $\phi(x)=H_p(x)$ which naturally fulfills the requirement $(\star)$ and we get $\sqrt{n}\, \mathcal{R}_{n,p}\to 0$ in probability. Hence, by Slutsky Lemma, one recovers (so far in the Gaussian case only) that 
\[
\left( \sqrt{n} \, p! \, \mathbb{E}_X\left[e_{n,p}(X)\right]\right)_{3\leq p\leq Q}\xrightarrow[n\to\infty]{\text{Law}}~\mathcal{N}(0,\text{Diag}(\sigma_3^2, \ldots, \sigma_Q^2))
\]
and in the same way, since the vectors have the same law under $\mathbb P$
\[
\left( \sqrt{n} \, p! \, \mathbb{E}_X\left[\widetilde{e}_{n,p}(X)\right]\right)_{3\leq p\leq Q}\xrightarrow[n\to\infty]{\text{Law}}~\mathcal{N}(0,\text{Diag}(\sigma_3^2, \ldots, \sigma_Q^2)).
\]
\noindent
\underline{The general case}: Now, let us go back to the general case of i.i.d. non-Gaussian coefficients that are centered with unit variance and a bounded sixth moment. Recall that, as computed in the end of Section \ref{sec.esti}, we have then the uniform bound 
\[
n p!^2 \mathbb{E}\left[\mathbb{E}_X\left[e_{n,p}(X)\right]^2\right] = n p!^2 \mathbb{E}\left[\mathbb{E}_X\left[\widetilde{e}_{n,p}(X)\right]^2\right] \leq p p!.
\]
Given the convergence established just above in the Gaussian framework and relying on Theorem 1.2 of \cite{nourdin2010invariance}, due to the homogeneous nature of the sums, we then automatically obtain the invariance principle
\[
\left( \sqrt{n} \, p! \, \mathbb{E}_X\left[\widetilde{e}_{n,p}(X)\right]\right)_{3\leq p\leq Q}\xrightarrow[n\to\infty]{\text{Law}}~\mathcal{N}(0,\text{Diag}(\sigma_3^2, \ldots, \sigma_Q^2)),
\]
and again since the sequences of the same law under $\mathbb P$
\[
\left( \sqrt{n} \, p! \, \mathbb{E}_X\left[e_{n,p}(X)\right]\right)_{3\leq p\leq Q}\xrightarrow[n\to\infty]{\text{Law}}~\mathcal{N}(0,\text{Diag}(\sigma_3^2, \ldots, \sigma_Q^2)).
\]
\noindent
\underline{Incorporating the second chaos component}: In view of Equation 
\eqref{eq.key22}, we are left to incorporate the contribution coming from the second chaos projection.  
First, as noted in Remark \ref{rem.first}
\begin{eqnarray*}
\sqrt{n}~c_2(\phi)\mathbb{E}_X\left[H_2\left(S_n(X)\right)\right]&=&\sqrt{n}~c_2(\phi)\mathbb{E}_X\left[\left(S_n(X)^2-1\right)\right]\\
&=&\frac{c_2(\phi)}{\sqrt{n}} \sum_{k=1}^n \left(\frac{a_k^2+b_k^2}{2}-1\right).
\end{eqnarray*}
\noindent
By the standard central limit Theorem for independent and identically distributed random variables, one then naturally deduce the convergence in distribution 
\[
\sqrt{n}~c_2(\phi)\mathbb{E}_X\left[H_2\left(S_n(X)\right)\right]
\xrightarrow[n\to\infty]~\mathcal{N}\left(0,\frac{c_2(\phi)^2}{2} \left(\mathbb{E}[a_1^4]-1\right)\right).
\]
We are now left to establish the joint convergence of the two vectors $\sqrt{n} \,\mathbb{E}_X\left[H_2\left(S_n(X)\right)\right]$ and $\left(\sqrt{n} \, p! \, \mathbb{E}_X\left[e_{n,p}(X)\right]\right)_{3\leq p\leq Q}$. To do so, it could be tempting to apply the multidimensional Theorem 7.10 of \cite{nourdin2010invariance}, but we have to pay attention here to the fact that the variables $\{a_k, b_k\}$ and $\{(a^2_k+b_k^2)/2-1\}$ are not independent, so that the result does not apply. Instead, we can adapt the original proof, making use of the standard Lindeberg's trick consisting of replacing the variables one after the other by Gaussian ones. Doing so, one gets the following proposition, whose detailed proof is given in Section \ref{sec.lindeberg}.
\begin{prop}\label{prop.joint}Under $\mathbb P$, we have the following joint  convergence in distribution as $n$ goes to infinity
\[
\left( \frac{1}{\sqrt{n}} \sum_{k=1}^n \left(\frac{a_k^2+b_k^2}{2}-1\right), \, \left(\sqrt{n} \, p! \, \mathbb{E}_X\left[e_{n,p}(X)\right]\right)_{3\leq p\leq Q}\right)\xrightarrow[n\to\infty]{\text{Law}}~\mathcal{N}(0,\Sigma),
\]
with 
\[
\Sigma:=\mathrm{Diag}\left( \frac{\mathbb E[a_1^4]-1}{2} , \sigma_3^2, \ldots, \sigma_Q^2\right).
\]
\end{prop}
\noindent
We can now complete the proof of Theorem \ref{Main-pas-Gaussian}. By Proposition \ref{prop.joint}, we get that
\begin{eqnarray*}
&&\sqrt{n}~c_2(\phi)\mathbb{E}_X\left[H_2\left(S_n(X)\right)\right]+\sqrt{n}\sum_{p=3}^Q~p! c_p(\phi)\mathbb{E}_X\left[e_{n,p}(X)\right]\\
&&\xrightarrow[n\to\infty]~\mathcal{N}\left(0,\frac{c_2(\phi)^2}{2} \left(\mathbb{E}(a_1^4)-1\right)+\sum_{k=3}^Q c_k(\phi)^2\sigma_k^2\right).
\end{eqnarray*}
Letting $Q\to\infty$ and relying on the tail estimate \eqref{Juste-nbre-fini-de-termes}, we then obtain that 
\begin{eqnarray*}
&&\sqrt{n}~c_2(\phi)\mathbb{E}_X\left[H_2\left(S_n(X)\right)\right]+\sqrt{n}\sum_{p=3}^\infty~p! c_p(\phi)\mathbb{E}_X\left[e_{n,p}(X)\right]\\
&&\xrightarrow[n\to\infty]~\mathcal{N}\left(0,\frac{c_2(\phi)^2}{2} \left(\mathbb{E}(a_1^4)-1\right)+\sum_{k=3}^{\infty} c_k(\phi)^2 \sigma_k^2 \right).
\end{eqnarray*}
Note that, from Remark \ref{rem.first} and the expression of the limit variance in Theorem \ref{Gaussian-hermite-clt}, we have ${}{\sigma_2^2=1}$ so that 
the above variance can be indeed rewritten 
\[
\begin{array}{ll}
\displaystyle{\Sigma_\phi^2} & :=\displaystyle{\frac{c_2(\phi)^2}{2} \left(\mathbb{E}(a_1^4)-1\right)+\sum_{k=3}^{\infty} c_k(\phi)^2 \sigma_k^2 =\frac{c_2(\phi)^2}{2} \left(\mathbb{E}(a_1^4)-3\right)+\sum_{k=2}^{\infty} c_k(\phi)^2 \sigma_k^2 }\\
\\
& \displaystyle{=\frac{c_2(\phi)^2}{2} \left(\mathbb{E}(a_1^4)-3\right)+\sigma_\phi^2},
\end{array}
\]
which indeed completes the proof of Theorem \ref{Main-pas-Gaussian}.

\section{Proofs of technical estimates} \label{sec.tech}
In this last section, we give the rather technical proofs of the lemmas and theorems stated in Section \ref{sec.esti}, as well as the proof of the last Proposition \ref{prop.joint} on the joint central convergence.

\subsection{Proof of Lemma \ref{lem.reste0}}

We first notice that the term $\mathbb E_X\left[ e_{n,p}(X) (1-N_{n,2}(X)) \right]$ can be made explicit as follows
\[
\mathbb E_X\left[ e_{n,p}(X) (1-N_{n,2}(X)) \right] =\frac{1}{n^{1+p/2}}  \sum_{k_0=1}^n \sum_{\substack{k_1 \ldots, k_p \in [|1, n|] \\ k_1 < \ldots < k_p}} \mathbb E_X \left[ A_{\mathbf k}(X)\right]
\]
where we have set $\mathbf k=(k_0,k_1, \ldots, k_p)$ and 
\[
A_{\mathbf k}(X):=\left[ (a_{k_0} \cos(k_0 X)+b_{k_0} \sin(k_0 X))^2-1\right]\prod_{i=1}^p  (a_{k_i} \cos(k_i X)+b_{k_i} \sin(k_i X)).
\]
In particular, taking the square and the expectation with respect to $\mathbb P$, we get
\begin{equation}\label{eq.doublesum}
\mathbb E \left[ \mathbb E_X\left[ e_{n,p}(X) (1-N_{n,2}(X)) \right]^2 \right] 
=\frac{1}{n^{2+p}}   \sum_{\substack{1 \leq k_0, l_0 \leq n \\ k_1 \ldots, k_p \in [|1, n|] \\ l_1, \ldots, l_p \in [|1, n|] \\ k_1 < \ldots < k_p \\ l_1 < \ldots < l_p}}  \mathbb E_{\mathbb P \otimes \mathbb P_X \otimes \mathbb P_Y} \left[ A_{\mathbf k}(X) A_{\mathbf l}(Y)\right].
\end{equation}
The next lemma establishes the decorrelation of $A_{\mathbf k}(X)$ and $A_{\mathbf l}(Y)$ whenever multi-indexes $\mathbf k$ and $\mathbf l$ are different. It plays a crucial role in the proof of Lemma \ref{lem.reste0}.
\begin{lem}\label{lem.diag}
If $\mathbf k \neq \mathbf l$, then  $\mathbb E_{\mathbb P } \left[ A_{\mathbf k}(X) A_{\mathbf l}(Y) \right]=0$.
\end{lem}
\begin{proof}[Proof of Lemma \ref{lem.diag}]
Let us first show that $\mathbb E_{\mathbb P } \left[ A_{\mathbf k}(X) A_{\mathbf l}(Y) \right]=0$ as soon as we have $(k_1, \ldots, k_p) \neq (l_1, \ldots, l_p)$. 
Indeed, suppose that there exists $1 \leq r \leq n$ such that $k_{r} \notin \{l_j, \, 1 \leq j\leq p\}$. Then, in the case where we further assume that $k_0, l_0 \notin \{k_i, l_j, \, 1 \leq i,j\leq p\}$, by independence of the coefficients, one can write
\[
A_{\mathbf k}(X) A_{\mathbf l}(Y)  = (a_{k_r} \cos(k_r X)+b_{k_r} \sin(k_r X)) B
\]
where $B$ is integrable and independent of $(a_{k_r},b_{k_r})$. In particular, we get
\[
\mathbb E_{\mathbb P } \left[ A_{\mathbf k}(X) A_{\mathbf l}(Y)\right]= \mathbb E [(a_{k_r} \cos(k_r X)+b_{k_r} \sin(k_r X)) ]\mathbb E[B]=0.
\]
In the same way, if $k_r=k_0 \neq l_0$, we can decompose 
\[
A_{\mathbf k}(X) A_{\mathbf l}(Y) = \left[(a_{k_r} \cos(k_r X)+b_{k_r} \sin(k_r X))^3-(a_{k_r} \cos(k_r X)+b_{k_r} \sin(k_r X))\right] B'
\]
where $B'$ is integrable and independent of $(a_{k_r},b_{k_r})$. In particular, since the coefficients are symmetric, we get again $\mathbb E_{\mathbb P } \left[ A_{\mathbf k}(X) A_{\mathbf l}(Y) \right]=0$. In the last case where $k_r=k_0 = l_0$, we have the similar decomposition 
\[
A_{\mathbf k}(X) A_{\mathbf l}(Y)  = (a_{k_r} \cos(k_r X)+b_{k_r} \sin(k_r X))\left[ (a_{k_r} \cos(k_r X)+b_{k_r} \sin(k_r X))^2-1\right]^2  B''
\]
where $B''$ is again independent of $(a_{k_r},b_{k_r})$ and the same conclusion holds. Let us now suppose that $(k_1, \ldots, k_p) = (l_1, \ldots, l_p)$ and $k_0 \neq l_0$. In the case where we have moreover $k_0 \notin \{k_i= l_i, \, 1 \leq i\leq p\}$, we have as above 
\[
A_{\mathbf k}(X) A_{\mathbf l}(Y) =\left[ (a_{k_0} \cos(k_0 X)+b_{k_0} \sin(k_0 X))^2-1\right] C
\]
where $C$ is integrable and independent of $(a_{k_0},b_{k_0})$ so that $\mathbb E_{\mathbb P } \left[ A_{\mathbf k}(X) A_{\mathbf l}(Y) \right]=0$. In the last case where both $k_0, l_0 \in \{k_i= l_i, \, 1 \leq i\leq p\}$, one can write similarly 
\[
A_{\mathbf k}(X) A_{\mathbf l}(Y) = \left[(a_{k_0} \cos(k_0 X)+b_{k_0} \sin(k_0 X))^3-(a_{k_0} \cos(k_0 X)+b_{k_0} \sin(k_0 X))\right] C'
\]
where $C'$ is integrable and independent of $(a_{k_0},b_{k_0})$, hence the result.
As a conclusion we get $\mathbb E_{\mathbb P } \left[ A_{\mathbf k}(X) A_{\mathbf l}(Y) \right]=0$ if $\mathbf k \neq \mathbf l$.
\end{proof}
\noindent
By Lemma \ref{lem.diag}, Equation \eqref{eq.doublesum} simplifies into 
\begin{equation}\label{eq.doublesum2}
\mathbb E \left[ \mathbb E_X\left[ e_{n,p}(X) (1-N_{n,2}(X)) \right]^2 \right] 
=\frac{1}{n^{2+p}}   \sum_{\substack{1 \leq k_0 \leq n \\ k_1 \ldots, k_p \in [|1, n|] \\ k_1 < \ldots < k_p}}  \mathbb E_{\mathbb P \otimes \mathbb P_X \otimes \mathbb P_Y} \left[ A_{\mathbf k}(X) A_{\mathbf k}(Y)\right],
\end{equation}
and we are left with estimating the term $\mathbb E_{\mathbb P \otimes \mathbb P_X \otimes \mathbb P_Y} \left[ A_{\mathbf k}(X) A_{\mathbf k}(Y)\right]$. 
\par
\medskip
\noindent
{\bf Case 1: }Let us first suppose that $k_0 \notin \{k_i, \, 1 \leq i\leq p\}$. Taking the expectation with respect to $\mathbb P$, we get by independence of the coefficients
\[
\begin{array}{ll}
\mathbb E_{\mathbb P} \left[ A_{\mathbf k}(X)A_{\mathbf k}(Y)\right] \\
\\
= \mathbb E_{\mathbb P} \left[\left[ (a_{k_0} \cos(k_0 X)+b_{k_0} \sin(k_0 X))^2-1\right]\left[ (a_{k_0} \cos(k_0 Y)+b_{k_0} \sin(k_0 Y))^2-1\right]\right]\\
\\
\displaystyle{\times \prod_{i=1}^p  \mathbb E_{\mathbb P} \left[ (a_{k_i} \cos(k_i X)+b_{k_i} \sin(k_i X))(a_{k_i} \cos(k_i Y)+b_{k_i} \sin(k_i Y))\right]}.
\end{array}
\]
A direct computation then yields the expression
\begin{equation}\label{eq.AkAk}
\begin{array}{ll}
\mathbb E_{\mathbb P} \left[ A_{\mathbf k}(X)A_{\mathbf k}(Y)\right] 
& = \displaystyle{\left( \frac{\mathbb E_{\mathbb P}[a_{k_0}^4] -1}{2}\right) \prod_{i=1}^p \cos(k_i (X-Y))}\\
\\
& \displaystyle{+  \left( \frac{\mathbb E_{\mathbb P}[a_{k_0}^4] +1}{4}\right) \cos(2k_0 (X-Y))\prod_{i=1}^p \cos(k_i (X-Y))}\\
\\
& \displaystyle{+\left(\frac{ \mathbb E_{\mathbb P}[a_{k_0}^4] -3}{4} \right)\cos(2k_0 (X+Y))\prod_{i=1}^p \cos(k_i (X-Y))}.
\end{array}
\end{equation}
{\bf Case 2:} Let us now suppose that $k_0 \in \{k_i, \, 1 \leq i\leq p\}$. 
A rather fastidious but direct computation then yields similarly
\begin{equation}\label{eq.AkAkbis}
\begin{array}{l}
\mathbb E_{\mathbb P} \left[ A_{\mathbf k}(X)A_{\mathbf k}(Y)\right] = \\
\\ 
\displaystyle{\frac{1}{32} \left( 2 \,\mathbb E_{\mathbb P}[a_{k_0}^6] +  6 \, \mathbb E_{\mathbb P}[a_{k_0}^4]  \right) \cos(3k_0( X- Y))  \prod_{\substack{1 \leq i \leq p \\ k_i \neq k_0}} \cos(k_i (X-Y))}\\
\\
\displaystyle{+ \frac{1}{32}\left( 18 \,\mathbb E_{\mathbb P}[a_{k_0}^6] +6 \,\mathbb E_{\mathbb P}[a_{k_0}^4] -16 \right) \cos( k_0(X -  Y)) \prod_{\substack{1 \leq i \leq p \\ k_i \neq k_0}} \cos(k_i (X-Y))}\\
\\
 \displaystyle{+ \frac{1}{32}\left( 6 \,\mathbb E_{\mathbb P}[a_{k_0}^6] -38\, \mathbb E_{\mathbb P}[a_{k_0}^4] +24 \right)\cos( 3k_0 X +  k_0 Y)\prod_{\substack{1 \leq i \leq p \\ k_i \neq k_0}} \cos(k_i (X-Y))}\\
\\
\displaystyle{+ \frac{1}{32} \left( 6 \,\mathbb E_{\mathbb P}[a_{k_0}^6] -38\, \mathbb E_{\mathbb P}[a_{k_0}^4] +24 \right)\cos( k_0X +  3 k_0Y) \prod_{\substack{1 \leq i \leq p \\ k_i \neq k_0}} \cos(k_i (X-Y))}.
\end{array}
\end{equation} 
Taking the expectation with respect to $\mathbb P_X \otimes \mathbb P_Y$, we have 
\[
\mathbb E_{\mathbb P_X \otimes \mathbb P_Y}\left[ \prod_{i=1}^p \cos(k_i (X-Y))\right] = \mathbb E_{\varepsilon}\left[ \mathds{1}_{\sum_{i=1}^p \varepsilon_i k_i=0} \right]
\]
where $\varepsilon=(\varepsilon_1, \ldots, \varepsilon_p)$ is a $p-$uplet of independent Rademacher random variables with $\mathbb P(\varepsilon_1=1)=\mathbb P(\varepsilon_1=-1)=1/2$. In particular, we have
\[
\begin{array}{ll}
\displaystyle{\frac{1}{n^{p}}   \sum_{\substack{ k_1 \ldots, k_p \in [|1, n|] \\ k_1 < \ldots < k_p}} \mathbb  E_{\mathbb P_X \otimes \mathbb P_Y}\left[ \prod_{i=1}^p \cos(k_i (X-Y))\right]}  &= \displaystyle{\mathbb E_{\varepsilon}\left[ \frac{1}{n^{p}}   \sum_{\substack{ k_1 \ldots, k_p \in [|1, n|] \\ k_1 < \ldots < k_p}} \mathds{1}_{\sum_{i=1}^p \varepsilon_i k_i=0} \right]}\\
\\
& \leq  \displaystyle{\frac{1}{n^{p}}   \sum_{\substack{ k_1 \ldots, k_{p-1} \in [|1, n|] \\ k_1 < \ldots < k_{p-1}}} 1 = \frac{1}{n(p-1)!}}.
\end{array}
\]
In the same manner, we get 
\[
\mathbb E_{\mathbb P_X \otimes \mathbb P_Y}\left[ \cos(2k_0 (X-Y)) \prod_{i=1}^p \cos(k_i (X-Y))\right] = \mathbb E_{\varepsilon}\left[ \mathds{1}_{2\varepsilon_0 k_0+ \sum_{i=1}^p \varepsilon_i k_i=0} \right],
\]
where this time $\varepsilon=(\varepsilon_0,\varepsilon_1, \ldots, \varepsilon_p)$ is a $(p+1)-$uplet of independent Rademacher random variables, from which we deduce that 
\[
\begin{array}{l}
\displaystyle{\frac{1}{n^{1+p}}   \sum_{\substack{1 \leq k_0 \leq n \\ k_1 \ldots, k_p \in [|1, n|] \\ k_1 < \ldots < k_p}} \mathbb E_{\mathbb P_X \otimes \mathbb P_Y}\left[ \cos(2k_0 (X-Y)) \prod_{i=1}^p \cos(k_i (X-Y))\right] \leq \frac{1}{n p!}}
\end{array}
\]
Proceeding similarly for each of the terms in Equations \eqref{eq.AkAk} and \eqref{eq.AkAkbis}, and coming back to Equation \eqref{eq.doublesum2}, we obtain that there exists a universal constant $C=C\left( \mathbb E_{\mathbb P}[a_{k_0}^4], \mathbb E_{\mathbb P}[a_{k_0}^6] \right)$ such that 
\begin{equation}\label{eq.doublesum3}
\begin{array}{l}
n \, \mathbb E \left[ \mathbb E_X\left[ e_{n,p}(X) (1-N_{n,2}(X)) \right]^2 \right]   \leq \frac{C}{n (p-1)!}.
\end{array}
\end{equation}

\subsection{Proof of Lemma \ref{lem.reste1}}

We now proceed with the proof of Lemma \ref{lem.reste1}.
We first notice that
\begin{eqnarray*}
N_{n,2}(X)&=&\frac{1}{n}\sum_{k=1}^n \left(a_k \cos(kX)+b_k\sin(kx)\right)^2\\
&=&\frac{1}{n} \sum_{k=1}^n a_k^2 \cos^2(kX)+b_k^2 \sin^2(kX)+a_k b_k \sin(2kX)\\
&=&\frac{1}{n}\sum_{k=1}^n \frac{a_k^2+b_k^2}{2}+\frac{1}{n}\sum_{k=1}^n\frac{a_k^2-b_k^2}{2}\cos(2 k X) +\frac{1}{n}\sum_{k=1}^n a_k b_k \sin(2 k X).
\end{eqnarray*}
To control the term $|1-N_{2,n}(X)|$, we will use different laws of iterated logarithm. First, using the standard law of iterated logarithm for i.i.d. variables, and since we have trivially $\log \circ \log <<\log$, we have that $\mathbb P-$almost surely, there exists a constant $C=C(\omega)$ such that
\[
\left|\sum_{k=1}^n \frac{a_k^2+b_k^2}{2}-n\right|\le C(\omega) \sqrt{n\log(n)}.
\]
Two handle the two remaining terms, we use the law of iterated logarithm for random trigonometric polynomials, as established in  \cite[Thm 4.3.1 p.270]{salem1954}.
Let us recall it below. Take $(r_k)_{k\ge 1}$ some deterministic sequence, and $(\epsilon_k)_{k\ge 1}$ an independent sequence of Rademacher symmetric random variables. Then, almost surely
\begin{equation}\label{SZ-max}
\limsup_{n\to\infty}\frac{\displaystyle{\sup_{\theta\in[0,2\pi]}}\left|\sum_{k=1}^n \epsilon_k r_k \cos(k\theta)\right|}{\sqrt{\log(n)~\sum_{k=1}^n r_k^2}}\le 2.
\end{equation}
Besides, as noticed in \cite[p 265]{salem1954}, this result can be extended in the same way to linear combinations of sine functions. 
Now, if $(X_k)$ is a sequence of i.i.d. symmetric random variables, one can always identify the laws of $(X_k)$ and $(\epsilon_k X_k)$ where $(\epsilon_k)$ is an independent sequence of i.i.d. symmetric Rademacher variables. Applying \eqref{SZ-max} with respect to the Rademacher variables conditionally to $(X_k)$, we then get that almost surely, there exists a constant $C=C(\omega)$ such that 
\begin{equation}\label{SZ-max2}
\sup_{\theta\in[0,2\pi]}\left|\sum_{k=1}^n X_k \cos(k\theta)\right| \leq C(\omega) \sqrt{\log(n)} \sqrt{\sum_{k=1}^n X_k^2},
\end{equation}
and similarly for random combinations of sine functions.
In our context,  note that the random variables $\{a_k b_k\}_{k\ge 1}$ and $\{a_k^2-b_k^2\}_{k\ge 1}$ are symmetric, therefore applying \eqref {SZ-max2} combined with the strong law of large numbers, we get that almost surely, there exists a constant $C=C(\omega)>0$ such that
\[
\sup_{x\in[0,2\pi]}\left|\sum_{k=1}^n\frac{a_k^2-b_k^2}{2}\cos(2 k x)\right|\le C(\omega) \sqrt{\log(n)} \sqrt{\sum_{k=1}^n\left(\frac{a_k^2-b_k^2}{2}\right)^2}\le C(\omega) \sqrt{n \log(n)},
\]
and similarly 
\[
\sup_{x\in[0,2\pi]}\left| \sum_{k=1}^n a_k b_k \sin(2 k x)\right|\le C(\omega)\sqrt{n \log(n)}.
\]
Gathering all theses bounds implies that $\mathbb P-$almost surely
\begin{equation}\label{bound-N2}
\sup_{x\in[0,2\pi]}\left|1-N_{n,2}(x)\right|\le C(\omega) \sqrt{\frac{\log(n)}{n}}.
\end{equation}
Hence, for all $j\ge 2$ the latter implies that

\begin{equation}\label{Lem-tech-jsup2}
\left|\mathbb{E}_X\left[\left(1-N_{n,2}(X)\right)^j e_{n,p}(X)\right]\right|\le \mathbb{E}_X\left[\left|e_{n,p}(X)\right|\right] C(\omega)^j \sqrt{\frac{\log(n)}{n}}^j.
\end{equation}
Thus, one is left to estimate the remaining term $ \mathbb{E}_X\left[\left|e_{n,p}(X)\right|\right]$. We simply notice that $\mathbb{E}_{\mathbb{P}\otimes\mathbb{P}_X}\left[e_{n,p}(X)^2\right]\le \frac{1}{p!}$, indeed by independence of $\{a_{k},b_{k}\}_{k\ge 1}$ one has uniformly in $X$
\begin{equation}
\mathbb{E}_{\mathbb{P}}\left[e_{n,p}(X)^2\right]=\frac{1}{n^p}\sum_{k_1<k_2<\cdots<k_p} \prod_{i=1}^p \mathbb{E}_{\mathbb{P}}\left[\left(a_{k_i}\cos(k_i X)+b_{k_i}\sin(k_i X)\right)^2\right]\leq  \frac{1}{p!}.\label{esperance-enp-carre}
\end{equation}
Since we have assumed the finiteness of the sixth-moment of the random variables $\{a_k,b_k\}_{k\ge 1}$ we may use hypercontractivity results for the homogeneous sum $e_{n,p}(X)$, see for example \cite[Propositions 3.11 and 3.12]{mossel2010noise}, which entail the fact that $L^2$ and $L^4$ norms are equivalent for this sequence of random variables. More precisely, Proposition 3.12 of the latter reference guarantees that for some constant $C>0$ and uniformly in the variable $X$

\begin{equation}\label{hypercon}
\mathbb{E}_{\mathbb{P}}\left[e_{n,p}(X)^4\right]\le \frac{C^p}{p!^2}.
\end{equation}
Note that, one could have obtained the same kind of estimate by a simple computation of $\mathbb{E}_{\mathbb{P}}\left[e_{n,p}(X)^4\right]$, but we used the above hypercontrativity argument in order to avoid computations.

\noindent
By Markov inequality, we have then
\[
\begin{array}{ll}
 \displaystyle{\mathbb{P}\left(\sum_{p} p!^{\frac 1 4} \mathbb{E}_X\left[\left|e_{n,p}(X)\right|\right]>n^{\frac{3}{8}}\right) \le \frac{1}{n^{\frac{3}{2}}}\mathbb{E}_{\mathbb{P}}\left[\left(\sum_{p} p!^{\frac 1 4} \mathbb{E}_X\left[\left|e_{n,p}(X)\right|\right]\right)^4\right]}\\
= \displaystyle{\frac{1}{n^{\frac 3 2}}\sum_{p_1,p_2,p_3,p_4}
\left(\prod_{i=1}^4 p_i !\right)^{1/4} \mathbb E\left[\prod_{i=1}^4 \mathbb{E}_X\left[\left|e_{n,p_i}(X)\right|\right]\right]}\\
\\
 \displaystyle{ \leq \frac{1}{n^{\frac 3 2}}\sum_{p_1,p_2,p_3,p_4}
\left(\prod_{i=1}^4 p_i !\right)^{1/4} \prod_{i=1}^4 \mathbb{E}_{\mathbb{P}\otimes\mathbb{P}_X}\left[e_{n,p_i}(X)^4\right]^{\frac{1}{4}}}\\
\\
 \stackrel{\eqref{hypercon}}{\le}  \displaystyle{\frac{1}{n^{\frac 3 2}}
\sum_{p_1,p_2,p_3,p_4}\left(\prod_{i=1}^4 p_i !\right)^{1/4} \left( \prod_{i=1}^4 \frac{C^{p_i}}{p_i!^2}  \right)^{1/4}\le \frac{C}{n^{\frac 3 2}}}.
\end{array}
\]
Hence, relying on Borel--Cantelli Lemma, there exists a constant $C(\omega)>0$ such that for every $p,n\ge 1$ we have
\begin{equation}\label{BCL}
\mathbb{E}_X\left[\left|e_{n,p}(X)\right|\right]\le C(\omega)\frac{n^{\frac 3 8}}{p!^{\frac 1 4}}.
\end{equation}
Combining \eqref{BCL} and \eqref{Lem-tech-jsup2} gives that, $\mathbb P-$almost surely, there exists a constant $C(\omega)>0$ such that for every $j\ge 2$ and every $n,p\ge 1$, we have
$$\sqrt{n}\left|\mathbb{E}_X\left[\left(1-N_{n,2}(X)\right)^j e_{n,p}(X)\right]\right|\le C(\omega)^{j}\sqrt{\log(n)}^j \frac{n^{\frac{7}{8}-\frac{j}{2}}}{p!^{\frac 1 4}}.$$
which brings the desired conclusion.

\subsection{Proof of Lemma \ref{bound.rnp}}
We now give the proof of Lemma \ref{bound.rnp} which provides bounds for the remainder term $R_{n,p}$ and tracks the dependence on $p$. Recall that by convention $R_{n,1}=R_{n,2}=0$. Let us first give uniform bounds on the terms $N_{1,n}(X)$ and $N_{3,n}(X)$. Here and in the sequel, $C(\omega)$ stands for a constant only depending on $\{a_k,b_k\}$ and which may change from line to line.
We proceed as in the proof of Lemma \ref{lem.reste1}, using the law of iterated logarithm established by Salem and Zygmund. Namely, using the estimate \eqref{SZ-max2}, since the variables $\{a_k,b_k\}_{k\ge 1}$ are symmetric, there exists a constant $C=C(\omega)$ such that $\mathbb P-$almost surely 
\[
\sup_{x\in[0,2\pi]}\left| \sum_{k=1}^n a_k \cos(kx) +b_k \sin(k x)\right|\le C(\omega) \sqrt{\log(n)}\sqrt{\sum_{k=1}^n a_k^2+b_k^2} \leq  C(\omega)\sqrt{n \log(n)}.
\]
so that uniformly in $X$
\begin{equation}\label{eq.n1}
|N_{n,1}(X)|\leq C(\omega)\sqrt{\log(n)}.
\end{equation}
We deal with the term $|N_{n,3}(X)|$ in the same manner. Indeed, expanding the cube and linearising the trigonometric functions, we have 
\[
\begin{array}{ll}
\displaystyle{\sum_{k=1}^n \left(a_k \cos(kx) +b_k \sin(k x)\right)^3}  =  \displaystyle{ \sum_{k=1}^n a_k^3 \cos^3(kx) +\sum_{k=1}^nb_k^3 \sin^3(k x) }\\
\\
  \displaystyle{+ 3 \sum_{k=1}^na_k^2 b_k\cos^2(kx) \sin(k x)+ 3 \sum_{k=1}^na_k b_k^2\cos(kx) \sin^2(k x)}\\
\\
=  \displaystyle{\sum_{k=1}^n a_k^3 \times \frac{3\cos(kx)+\cos(3kx)}{4} +\sum_{k=1}^n b_k^3 \times \frac{3\sin(kx)-\sin(3kx)}{4} }\\
\\
   \displaystyle{+3 \sum_{k=1}^n a_k^2 b_k \times \frac{\sin(kx) - \sin(3kx)}{4} +3 \sum_{k=1}^n a_k b_k^2\times \frac{\cos(kx)-\cos(3kx)}{4}}.
\end{array}
\]
Then, noticing that all the variables $a_k^3, b_k^3, a_k^2 b_k$ and $a_k b_k^2$ are symmetric, we may apply the estimate \eqref{SZ-max2} for each of the eight terms above and we conclude that 
\[
\sup_{x\in[0,2\pi]}\left| \sum_{k=1}^n \left(a_k \cos(kx) +b_k \sin(k x)\right)^3\right|\le C(\omega) \sqrt{\log(n)}\left( \sqrt{\sum_{k=1}^n a_k^6+b_k^6+a_k^4 b_k^2 +a_k^2 b_k^4 } \right).
\]
As a result, uniformly in $X$, we get 
\begin{equation}\label{eq.n3}
|N_{n,3}(X)|\leq C(\omega)\frac{\sqrt{\log(n)}}{n}.
\end{equation}
By definition, for $p=3$, we have 
\[
R_{n,3}(X) =\frac{1}{3} N_{n,3}(X), \quad \text{so that we have indeed} \quad |R_{n,3}(X)| \leq C(\omega)\frac{\sqrt{\log(n)}}{n}.
\]
For $j\in\{4,5,6\}$, we simply write
\begin{equation}\label{eq.n456}
\left|N_{n,j}(X)\right|\le \frac{1}{n^{\frac j 2 -1}}\underbrace{\frac{1}{n}\sum_{k=1}^n (|a_k|+|b_k|)^j}_{\text{Law of large numbers}}\le \frac{C(\omega)}{n^{\frac{j-2}{2}}}\le\frac{C(\omega)}{n}.
\end{equation}
Injecting these estimates in the trivial upper bound
\begin{equation}\label{eq.boundrnp}
\left|R_{n,p}(X)\right| \leq \sum_{\substack{m_1+2m_2+\cdots+pm_p=p\\m_1\ge 0,\cdots,m_p\ge 0\\\ \exists j\ge 3, m_j>0}} \left(\prod_{j=1}^p \frac{1}{m_j! j^{m_j}}\right)\prod_{j=1}^p \left| N_{n,j}(X)\right|^{m_j} ,
\end{equation}
we get 
\[
\left|R_{n,4}(X)\right|\leq \frac{1}{3} \left|N_{n,1}(X)N_{n,3}(X)\right|+\frac{1}{4} \left|N_{n,4}(X)\right| \leq C(\omega)\frac{\log(n)}{n}.
\]
In the same way, using \eqref{bound-N2}, we get
\[
\begin{array}{ll}
\left|R_{n,5}(X)\right|& \displaystyle{\leq \frac{1}{6} \left|N_{n,1}(X)^2N_{n,3}(X)\right|+\frac{1}{6} \left|N_{n,2}(X)N_{n,3}(X)\right| }\\
& \displaystyle{+ \frac{1}{4} \left|N_{n,1}(X) N_{n,4}(X)\right|+\frac{1}{5} \left| N_{n,5}(X)\right|  \leq C(\omega)\frac{\sqrt{\log(n)}^3}{n}},
\end{array}
\]
and 
\[
\begin{array}{ll}
\left|R_{n,6}(X)\right|& \displaystyle{\leq \frac{1}{6} \left|N_{n,1}(X)N_{n,2}(X)N_{n,3}(X)\right|+\frac{1}{6} \left|N_{n,1}(X)^3N_{n,3}(X)\right| }\\
& \displaystyle{+ \frac{1}{8} \left|N_{n,3}(X)^2\right|+\frac{1}{8} \left|N_{n,2}(X) N_{n,4}(X)\right|+\frac{1}{8} \left|N_{n,1}(X)^2 N_{n,4}(X)\right|}\\
& \displaystyle{ +\frac{1}{5} \left|N_{n,1}(X) N_{n,5}(X)\right| +\frac{1}{6} \left| N_{n,6}(X)\right|  \leq C(\omega)\frac{\log(n)^2}{n}}.
\end{array}
\]
Let us now consider larger values of $p$.  As we have assumed a finite $6-$th moment for the random variables $\{a_k,b_k\}_{k\ge 1}$, using Markov inequality we get 
\[
\mathbb{P}\left(|a_k|\ge k^{\frac{1}{5}}\right)\le \frac{\mathbb{E}_{\mathbb{P}}\left[|a_1|^6\right]}{k^{\frac 6 5}}.
\] 
Hence, Borel--Cantelli Lemma ensures that, for some constant $C(\omega)>0$, we have $$\forall k\ge 1,\quad \,|a_k|+|b_k|\le C(\omega) k^{\frac{1}{5}}.$$
As a result, for every $j\ge 7$ the triangular inequality entails
\begin{eqnarray}
\left|N_{n,j}(X)\right|&=&\frac{1}{n^{\frac j 2}}\left|\sum_{k=1}^n\left(a_k\cos(kX)+b_k\sin(k X)\right)^j\right|
\leq \frac{C(\omega)^{j-6}}{n^{\frac j 2}}\sum_{k=1}^n k^{\frac{j-6}{5}} (|a_k|+|b_k|)^6 \nonumber\\
&\le& C(\omega)^{j-6} \frac{n^{\frac{j-6}{5}}}{n^{\frac {j}{ 2}-1}}\times \underbrace{\frac{1}{n}\sum_{k=1}^n (|a_k|+|b_k|)^6}_{\text{Law of large numbers}}\leq  C(\omega)^{j-6} \frac{1}{n^{\frac{1}{5}+\frac{3 j}{10}}}. \label{eq.ngrand}
\end{eqnarray}
Combining the estimates \eqref{bound-N2}, \eqref{eq.n1}, \eqref{eq.n3}, \eqref{eq.n456}, \eqref{eq.ngrand}, we get for $p \geq 7$
\begin{eqnarray}
\prod_{j=1}^p \left| N_{n,j}(X)^{m_j}\right|  & \displaystyle{\leq C(\omega)^{m_1+m_2+m_3} \frac{\log(n)^{\frac{m_1+m_3}{2}}}{n^{m_3}} \prod_{j=4}^6\left( \frac{C(\omega)}{n} \right)^{m_j} \prod_{j=7}^p\left( \frac{C(\omega)^{j-6}}{n^{\frac{1}{5} +\frac{3j}{10}}} \right)^{m_j} } \nonumber \\
\nonumber  \\
& \displaystyle{\leq C(\omega)^{2p} \log(n)^{\frac{p}{2}} \times \left( \frac{1}{n} \right)^{\sum_{j=3}^6 m_j + \frac{1}{5}\sum_{j=7}^p m_j + \frac{3}{10} \sum_{j=7}^p j m_j}}.\label{ngrandbis}
\end{eqnarray}
Notice that under the condition $\sum_{j=1}^p j m_j =p$ and $\exists j\geq 3, m_j>0$, we have always
\[
\sum_{j=3}^6 m_j + \frac{1}{5}\sum_{j=7}^p m_j + \frac{3}{10} \sum_{j=7}^p j m_j\geq 1
\]
and naturally if $\sum_{j=7}^p j m_j \geq p/2$
\[
\sum_{j=3}^6 m_j + \frac{1}{5}\sum_{j=7}^p m_j + \frac{3}{10} \sum_{j=7}^p j m_j\geq \frac{3p}{20}.
\]
Let us come back to the upper bound \eqref{eq.boundrnp} and remark that on the one hand
\[
\sum_{\substack{m_1+2m_2+\cdots+pm_p=p\\m_1\ge 0,\cdots,m_p\ge 0\\
\exists j\ge 3, m_j>0}}\prod_{j=1}^p \frac{1}{m_j! j^{m_j}}\leq 1,
\]
and on the other hand that, if $\sum_{j=7}^p j m_j < p/2$, we have 
$m_1+2m_2+\cdots+6m_6 > \frac{p}{2}$ so that there is an index $i\in\{1,2,3,4,5,6\}$ such that $i m_i\ge \frac{p}{12}$ and in that case
$$\prod_{j=1}^p \frac{1}{m_j! j^{m_j}}\le \frac{1}{\left[\frac{p}{72}\right]!} \prod_{j=7}^p \frac{1}{m_j! j^{m_j}}.$$
As a result, for $p \geq 7$, we obtain from Equations \eqref{eq.boundrnp} and \eqref{ngrandbis} that 
\[
\left|R_{n,p}(X)\right|  \displaystyle{\leq  C(\omega)^{2p} \log(n)^{\frac{p}{2}}  \times \left(  \frac{1}{n^{\frac{3p}{20}}} + \frac{1}{n} \frac{1}{\left[\frac{p}{72}\right]!}\sum_{\substack{m_1\ge 0,\cdots,m_p\ge 0\\ \sum_{j=1}^p j m_j=p \\ \sum_{j=7}^p j m_j<p/2}}  \prod_{j=7}^p \frac{1}{m_j! j^{m_j}}\right)}.
\]
Otherwise, the last sum can be upper bounded by 
\[
\begin{array}{ll}
\displaystyle{\sum_{\substack{m_1\ge 0,\cdots,m_p\ge 0 \\ \sum_{j=1}^p j m_j=p \\ \sum_{j=7}^p j m_j<p/2}}  \prod_{j=7}^p \frac{1}{m_j! j^{m_j}}} \displaystyle{ \leq \sum_{\substack{0 \leq m_i \leq p\\ 1\leq i \leq 6}}  \sum_{\substack{0 \leq m_i \leq +\infty \\ 7 \leq i \leq +\infty}}  \prod_{j=7}^p \frac{1}{m_j! j^{m_j}}} \\
\\
 \displaystyle{\leq (p+1)^6  \exp \left( \frac{1}{7} +\frac{1}{8} +\ldots +\frac{1}{p}\right) \leq C p^{7}}.
\end{array}
\]
Therefore, we conclude that 
\[
\begin{array}{ll}
\displaystyle{\left|R_{n,p}(X)\right| }& \leq  \displaystyle{C(\omega)^{2p} \log(n)^{\frac{p}{2}}  \times\left( \frac{1}{n^{\frac{3p}{20}}} + \frac{p^7}{n} \frac{1}{\left[\frac{p}{72}\right]!}\right)}\\
\\
&\leq \displaystyle{2C(\omega)^{2p} \log(n)^{\frac{p}{2}}  \times \max\left( \frac{1}{n^{\frac{3p}{20}}}, \frac{p^7}{n} \frac{1}{\left[\frac{p}{72}\right]!}\right)}.
\end{array}
\]

\subsection{Proof of Theorem \ref{Remainder-main}}

Let us now give the proof of Theorem \ref{Remainder-main} ensuring that the term $\mathcal R_{n,p}$ can be treated as a remainder term. First of all, using the content of Lemma \ref{bound.rnp} we may infer that

\[
\begin{array}{l}
\displaystyle{\sqrt{n}\left|\sum_{p=3}^\infty p!c_p(\phi) R_{n,p}(X)\right|\le  \frac{C(\omega)\log(n)^2}{\sqrt{n}}~\sum_{p=3}^6 p! |c_p(\phi)| }\\
\displaystyle{+C(\omega) \log(n)^{p/2}\sqrt{n}\sum_{p=7}^\infty p! |c_p(\phi)| \max\left(\frac{1}{n}\frac{p^7}{\lfloor\frac{p}{72}\rfloor !}\,\,,\,\, \frac{1}{n^\frac{3 p}{20}}\right)}.
\end{array}
\]
On the one hand, we notice that under the condition $(\star)$, the following serie is convergent
\[
\sum_{p=7}^\infty p! |c_p(\phi)| \frac{p^7}{\lfloor\frac{p}{72}\rfloor !}<\infty.
\]
On the other hand, by dominated convergence, we have also
\[
\lim_{n \to +\infty} \sqrt{n}\sum_{p=7}^\infty p! |c_p(\phi)|\frac{1}{n^{\frac{3p}{20}}}=0.
\]
Indeed, for $n$ and $p$ large enough, we have $n^{\frac{1}{2}-\frac{ 3p}{20}}\le A^p$ where $A$ is the positive constant appearing in condition $(\star)$. 
As a result, uniformly in the variable $X$, we have $\mathbb P-$almost surely
\[
\lim_{n \to +\infty} \sqrt{n}\left|\sum_{p=3}^\infty p!c_p(\phi) R_{n,p}(X)\right|=0.
\]
Note that, since all the bounds are uniform in $X\in[0,2\pi]$ the involved series are convergent and we can interchange limit and expectation to get

\begin{equation}\label{Tech-Theo4-1}
\sqrt{n} \, \mathbb{E}_X\left[ \sum_{p=3}^\infty p! c_p(\phi) R_{n,p}(X)\right]\xrightarrow[n\to\infty]{\text{a.s.}}~0.
\end{equation}
\medskip
\noindent
Then we fix an integer $Q\ge 1$ and we write

\begin{eqnarray*}
&&\mathbb{E}_X\left[\sqrt{n}\sum_{p=3}^Q p! c_p(\phi)\sum_{k=1}^{\lfloor \frac{p}{2}\rfloor}\frac{(-1)^k(1-N_{n,2}(X))^k}{2^k k!}\left(e_{n,p-2k}(X)-R_{n,p-2k}(X)\right)\right]\\
&=&\sqrt{n}\sum_{p=3}^Q \frac{p! c_p(\phi)}{2} \mathbb{E}_X\left[\left(N_{n,2}(X)-1 \right)e_{n,p-2}(X)\right]\\
&+&\sqrt{n}\sum_{p=4}^Q p! c_p(\phi) \sum_{k=2}^{\lfloor \frac{p}{2}\rfloor}\frac{(-1)^k }{2^k k!} \mathbb{E}\left[(1-N_{n,2}(X))^k e_{n,p-2k}(X)\right]\\
&+&\sqrt{n}\sum_{p=5}^Q p! c_p(\phi) \sum_{k=1}^{\lfloor \frac{p}{2}\rfloor}
(-1)^{k+1} \mathbb{E}_X\left[\frac{(1-N_{2,n}(X))^k}{2^k k!}R_{n,p-2k}(X)\right]\\
&:=& A_{n,Q}+B_{n,Q}+C_{n,Q}.
\end{eqnarray*}
Let us first deal with $A_{n,Q}$. Relying on Lemma \ref{lem.reste0} and Minkowski inequality one obtains (for a constant $C>0$ possibly changing from line to line)

\begin{eqnarray*}
\sqrt{\mathbb{E}_{\mathbb{P}}\left[A_{n,Q}^2\right]}&\le&\sum_{p=3}^Q \frac{p! |c_p(\phi)|}{2} \underbrace{\|\sqrt{n}\, \mathbb{E}_X\left[\left(N_{n,2}(X)-1 \right)e_{n,p-2}(X)\right]\|_2}_{=O\left(\sqrt{\frac{1}{\sqrt{n}(p-1)!}}\right)\,\,\text{using Lemma}\,\ref{lem.reste0}}\\
&\le& \frac{C}{n^{\frac 1 4}}\underbrace{\sum_{p=3}^\infty \frac{p! |c_p(\phi)|}{\sqrt{(p-1)!}}}_{<\infty\,\text{under}\,\,(\star)}.
\end{eqnarray*}
Letting $Q\to\infty$ and then $n\to \infty$ we obtain that $A_{n,\infty}$, which exists and belongs to $L^2(\mathbb{P})$, tends to zero in $L^2(\mathbb{P})$ and thus in probability.
\par
\medskip
Let us now deal with the term $B_{n,Q}$. Using Lemma \ref{lem.reste1} and the triangle inequality, we get (for some absolute constant $C(\omega)>0$ possibly changing from line to line)

\begin{eqnarray*}
\left|B_{n,Q}\right|&\le& \sum_{p=4}^Q p! |c_p(\phi)| \sum_{k=2}^{\lfloor \frac{p}{2} \rfloor}\frac{1}{2^k k!} \sqrt{n}\left|\mathbb{E}_X\left[(1-N_{n,2}(X))^k e_{n,p-2k}(X)\right]\right|\\
&\stackrel{\text{Lemma}\,\ref{lem.reste1}}{\le}&\sum_{p=4}^Q p! |c_p(\phi)| \sum_{k=2}^{\lfloor \frac{p}{2} \rfloor}\frac{C(\omega)^k}{2^k k!} \sqrt{\log(n)}^k \frac{n^{\frac{7}{8}-\frac{k}{2}}}{(p-2k)!^{\frac{1}{4}}}.
\end{eqnarray*}
To remove the logarithm in the numerator, we can simply argue that 
\[
\sqrt{\log(n)}\le C n^{\frac{1}{32}}\quad \text{and thus} \quad \sqrt{\log(n)}^k n^{\frac{7}{8}-\frac{k}{2}}\le C \, n^{\frac{7}{8}-\frac{15 k}{32}}.
\]
Then we distinguish two cases, either $k\ge \max\left(\lfloor \frac{p}{4} \rfloor,2\right)$  and we get 
\[
\frac{1}{2^k k!}  \frac{n^{\frac{7}{8}-\frac{15 k}{32}}}{(p-2k)!^{\frac{1}{4}}}\le \frac{1}{\lfloor \frac{p}{4}\rfloor !}n^{\frac{7}{8}-\frac{15 \lfloor \frac{p}{4} \rfloor }{32}}\le \frac{1}{\lfloor \frac{p}{4}\rfloor !}\frac{1}{n^{\frac{1}{16}}}.
\]
Otherwise $k\le \lfloor \frac{p}{4}\rfloor$ and we get (since $k\ge 2$ anyway) that
\[
\frac{1}{2^k k!}  \frac{n^{\frac{7}{8}-\frac{15 k}{32}}}{(p-2k)!^{\frac{1}{4}}}\le \frac{1}{n^{\frac{1}{16}}}\frac{1}{\lfloor\frac{p}{2}\rfloor!}.
\]
Gathering these two facts, one obtains
\[
\begin{array}{ll}
\displaystyle{\left|B_{n,Q} \right| }& \le \displaystyle{C(\omega) \sum_{p=4}^Q p! |c_p(\phi)|~~\frac{p}{2} C(\omega)^p\left(\frac{1}{\lfloor\frac{p}{4}\rfloor !}\frac{1}{n^{\frac{1}{16}}}+\frac{1}{n^{\frac{1}{16}}}\frac{1}{\lfloor\frac{p}{2}\rfloor!}\right)} \\
& \displaystyle{\le \frac{1}{n^{\frac{1}{16}}}\sum_{p=4}^Q \frac{p~p! C(\omega)^p}{\lfloor \frac{p}{4}\rfloor!} |c_p(\phi)|.}
\end{array}
\]
Following Remark \ref{rem.star}, the last series is convergent under the assumption $(\star)$ and thus we may let $Q\to\infty$ and get that $B_{n,\infty}$ is well-defined and tends to zero almost surely and hence in probability as $n$ goes to infinity.

It remains to handle the case of $C_{n,Q}$. Relying on Equation \eqref{bound-N2} established in the proof of Lemma \ref{lem.reste1}, we know that $\mathbb P-$almost surely 
\[
|1-N_{n,2}(X)|\le C(\omega) \frac{\sqrt{\log(n)}}{\sqrt{n}}.
\]
Therefore, we get uniformly in $X\in[0,2\pi]$

\begin{eqnarray*}
|C_{n,Q}|&\le& \sum_{p=5}^Q p! |c_p(\phi)|\sum_{k=1}^{\lfloor\frac{p}{2}\rfloor}\frac{\sqrt{n}~C(\omega)^k}{2^k k!} \left(\frac{\log(n)}{n}\right)^{\frac k 2}\sup_X\left|R_{n,p-2k}(X)\right|.
\end{eqnarray*}
Either $k\ge \lfloor \frac{p}{4}\rfloor$ and in this case, given that for any value of $p$ by Lemma \ref{bound.rnp}, we have
\[
\sup_X|R_{n,p}(X)|\le \frac{C(\omega)\log(n)^{\max(2,p/2)}}{n},
\]
we get, since for any exponent $\log(n)^{q}\ll\sqrt{n}$,
\[
\frac{\sqrt{n}~C(\omega)^k}{2^k k!} \left(\frac{\log(n)}{n}\right)^{\frac k 2}\sup_X\left|R_{n,p-2k}(X)\right|\le \frac{1}{\sqrt{n}} \frac{C(\omega)^p}{\lfloor\frac{p}{4} \rfloor!}.
\]
In the opposite case where $k< \lfloor \frac{p}{4}\rfloor$, and in the subcase where $\lfloor\frac{p}{2}\rfloor\ge 7$ that is to say if $p\ge 14$, we may use the second bound in  Lemma \ref{bound.rnp} to get
\[
\frac{\sqrt{n}~C(\omega)^k}{2^k k!} \left(\frac{\log(n)}{n}\right)^{\frac k 2}\sup_X\left|R_{n,p-2k}(X)\right|\le  C(\omega)^p \frac{\log(n)^{\frac{p-k}{2}}}{2^k k!}\left(\frac{1}{n} \frac{\lfloor\frac{p}{2}\rfloor^7}{\lfloor\frac{p}{144}\rfloor!}+\frac{1}{n^{\frac{3 p}{40}}}\right).
\]
Gathering these two cases for $p\ge 14$ and making the sum for $k\in\llbracket 1, \lfloor\frac{p}{2}\rfloor\rrbracket$ leads to
\begin{eqnarray*}
&&\sum_{k=1}^{\lfloor\frac{p}{2}\rfloor}\frac{\sqrt{n}~C(\omega)^k}{2^k k!} \left(\frac{\log(n)}{n}\right)^{\frac k 2}\sup_X\left|R_{n,p-2k}(X)\right|
\le \frac{p}{\sqrt{n}}\frac{C(\omega)^p}{\lfloor\frac{p}{4} \rfloor!}\\
&+&C(\omega)^p \left(\frac{1}{n} \frac{\lfloor\frac{p}{2}\rfloor^7}{\lfloor\frac{p}{144}\rfloor!}+\frac{1}{n^{\frac{3 p}{40}}}\right)e^{\frac{\sqrt{\log(n)}}{2}}.
\end{eqnarray*}
Then, using a dominated convergence argument which is based on the assumption $(\star)$ and Remark \ref{rem.star}, we may deduce that
\[
\sum_{p=14}^\infty p! |c_p(\phi)| \left(\frac{p}{\sqrt{n}}\frac{C(\omega)^p}{\lfloor\frac{p}{4} \rfloor!}+C(\omega)^p \left(\frac{1}{n} \frac{\lfloor\frac{p}{2}\rfloor^7}{\lfloor\frac{p}{144}\rfloor!}+\frac{1}{n^{\frac{3 p}{40}}}\right)e^{\frac{\sqrt{\log(n)}}{2}}\right)\xrightarrow[n\to\infty]{\text{a.s.}}~0.
\]
Finally, we can treat separately the finite number of terms $p\in \llbracket 5,13 \rrbracket$ and let first $Q$ go to infinity and then $n$ go to infinity in order to get that $C_{n,\infty}$ goes almost surely to zero as $n$ goes to infinity. This concludes the proof of Theorem \ref{Remainder-main}.

\subsection{Proof of Proposition \ref{prop.joint}}\label{sec.lindeberg}

We finally give here the non-trivial proof of Proposition \ref{prop.joint} on the joint convergence of the second and higher degree components. As mentioned in Section \ref{sec.noise}, the result can not be deduced directly form the ones of \cite{nourdin2010invariance} since the variables $\{a_k, b_k\}$ and $\{(a^2_k+b_k^2)/2-1\}$ are not independent. 
The proof below is largely inspired from the ones of the latter reference and it relies on the Lindeberg's replacement trick, for which we need to introduce a certain number of convenient notations. For any finite collection of variables $\mathbf{(x,y)}:=(x_k,y_k)_{1\le k\le n}$ we shall set
\begin{eqnarray*}
Q_2\left(\mathbf{x}\right)&=&\frac{1}{\sqrt{n}}\sum_{k=1}^n x_k  \quad \text{and} \quad  \forall j \in \{3,\cdots,p\} \\
Q_j\left(\mathbf{(x,y)}\right)&=&\frac{1}{n^{\frac{j}{2}}}\sum_{\substack{k_1 \ldots, k_j \in [|1, n|] \\ k_1 < \ldots < k_j}} \mathbb{E}\left[\prod_{i=1}^j \left( x_{k_i} \cos(k_i X)+y_{k_i} \sin(k_i X)\right)\right].
\end{eqnarray*}
Suppose that we are given the family of random variables $\{a_k,b_k\}$ and another independent family of i.i.d. standard Gaussian variables $\{g_k,\tilde{g}_k\}$.
Then, we shall adopt the following notations, $\forall i\in\llbracket 1,n-1\rrbracket$
\begin{eqnarray*}
\mathbf{(x,y)}_0&=&(a_k,b_k)_{1\le k \le n}~\text{$\to$ zero replacement}\\
\mathbf{(x,y)}_n&=&(g_k,\tilde{g}_k)_{1\le k\le n}~\text{$\to$ $n$ replacements}\\
\mathbf{(x,y)_i}&=&\{(g_1,\tilde{g}_1),\cdots,(g_i,\tilde{g}_i),(a_{i+1},b_{i+1}),\cdots,(a_n,b_n)\}~\text{$\to$ $i$ replacements},
\end{eqnarray*}
that is, we replace the $i-$first pair of variables by their Gaussian analogues.
Assuming that $\{a_k,b_k\}_{k\ge 1}$ do not follow a symmetric Rademacher distribution, we may set 
\[
\sigma^2=\text{Var}\left(\frac{a_1^2+b_1^2}{2}-1\right)=\frac{1}{2}\left(\mathbb{E}[a_1^4]-1\right)>0.
\]
Then, for some auxiliary sequence $\{z_k\}_{k\ge 1}$ of standard Gaussian random variables which is independent of $\{a_k,b_k,g_k,\tilde{g}_k\}_{k\ge 1}$, we need to introduce similarly, $\forall i\in\llbracket 1,n-1\rrbracket$
\begin{eqnarray*}
\mathbf{u}_0&=&\frac{1}{\sqrt{\sigma}}\left(\frac{a_k^2+b_k^2}{2}-1\right)_{1\le k \le n}\\
\mathbf{u}_n&=&(z_k)_{1\le k\le n}\\
\mathbf{u_i}&=&\left\{z_1,z_2,\cdots,z_i,\frac{1}{\sqrt{\sigma}}\left(\frac{a_{i+1}^2+b_{i+1}^2}{2}-1\right),\cdots,\frac{1}{\sqrt{\sigma}}\left(\frac{a_{n}^2+b_{n}^2}{2}-1\right)\right\}.
\end{eqnarray*}
\noindent
In the same way, for any finite collection of variables $\mathbf{(x,y)}$ and any $i\in\{1,\cdots,n\}$ we set $\mathbf{(x,y)^{(i)}}=\{(x_1,y_1),\cdots,(x_{i-1},y_{i-1}),\mathbf{(0,0)},(x_{i+1},y_{i+1}),\cdots,(x_n,y_n)\}$, that is to say the sequence that one obtains by setting $x_i=y_i=0$ in $\mathbf{(x,y)}$. We also set $\mathbf{u^{(i)}}$ the sequence obtained by setting $u_i=0$ in $\mathbf{u}$. In order to compare the distributions of homogeneous polynomial with Gaussian and non-Gaussian entries, let us compare their characteristic functions. With our ``iterated replacements'' notations, we thus want to estimate,  for any $t=(t_2,\cdots,t_p)\in\mathbb{R}^{p-1}$
\[
\begin{array}{ll}
\Delta_n(t) & \displaystyle{:=\mathbb{E}\left[e^{i t_2 Q_2(\mathbf{u}_0)+i \sum_{l=3}^p t_l Q_l \left(\mathbf{(x,y)}_0\right)}\right]-\mathbb{E}\left[e^{i t_2 Q_2(\mathbf{u}_n)+i \sum_{l=3}^p t_l Q_l \left(\mathbf{(x,y)}_n\right)}\right]}\\
\\
& \displaystyle{=\sum_{k=1}^n \mathbb{E}\left[e^{i t_2 Q_2(\mathbf{u}_{k-1})+i \sum_{l=3}^p t_l Q_l \left(\mathbf{(x,y)}_{k-1}\right)}\right]-\mathbb{E}\left[e^{i t_2 Q_2(\mathbf{u}_k)+i \sum_{l=3}^p t_l Q_l \left(\mathbf{(x,y)}_k\right)}\right].}\\
\\
& =\sum_{k=1}^n \left( A_k - B_k\right),
\end{array}
\]
where we have set
\begin{eqnarray*}
\hspace{-0.8cm}A_k&=&\mathbb{E}\left[e^{(i t_2 Q_2(\mathbf{u}_{k-1})+i \sum_{l=3}^p t_l Q_l \left(\mathbf{(x,y)}_{k-1}\right)}\right]-\mathbb{E}\left[e^{i t_2 Q_2(\mathbf{u}_{k-1}^{(k)})+i \sum_{l=3}^p t_l Q_l \left(\mathbf{(x,y)}_{k-1}^{(k)}\right)}\right]\\
\hspace{-0.8cm}B_k&=&\mathbb{E}\left[e^{i t_2 Q_2(\mathbf{u}_k)+i \sum_{l=3}^p t_l Q_l \left(\mathbf{(x,y)}_{k}\right)}\right]-\mathbb{E}\left[e^{i t_2 Q_2(\mathbf{u}_k^{(k)})+i \sum_{l=3}^p t_l Q_l \left(\mathbf{(x,y)}_{k}^{(k)}\right)}\right].
\end{eqnarray*}
To do so, we may write for any $l\in\{3,\cdots,p\}$ and $k\in\{1,\cdots,n\}$ that

\begin{eqnarray*}
Q_l\left(\mathbf{(x,y)}_{k}\right)&=&\underbrace{\frac{1}{n^{\frac{j}{2}}}\sum_{\substack{k_1 \ldots, k_j \in [|1, n|] \\ k_1 < \ldots < k_j\\ k\notin \{k_1,\cdots,k_j\}}} \mathbb{E}\left[\prod_{i=1}^j \left( x_{k_i} \cos(k_i X)+y_{k_i} \sin(k_i X)\right)\right]}_{=R_{l,k}}\\
&+&\underbrace{\frac{1}{n^{\frac{j}{2}}}\sum_{\substack{k_1 \ldots, k_j \in [|1, n|] \\ k_1 < \ldots < k_j\\ k\in \{k_1,\cdots,k_j\}}} \mathbb{E}\left[\prod_{i=1}^j \left( x_{k_i} \cos(k_i X)+y_{k_i} \sin(k_i X)\right)\right]}_{=S_{l,k}}\\
\end{eqnarray*}
and 
\begin{eqnarray*}
Q_l\left(\mathbf{(x,y)}_{k-1}\right)&=&\underbrace{\frac{1}{n^{\frac{j}{2}}}\sum_{\substack{k_1 \ldots, k_j \in [|1, n|] \\ k_1 < \ldots < k_j\\ k\notin \{k_1,\cdots,k_j\}}} \mathbb{E}\left[\prod_{i=1}^j \left( x_{k_i} \cos(k_i X)+y_{k_i} \sin(k_i X)\right)\right]}_{=\tilde{R}_{l,k}}\\
&+&\underbrace{\frac{1}{n^{\frac{j}{2}}}\sum_{\substack{k_1 \ldots, k_j \in [|1, n|] \\ k_1 < \ldots < k_j\\ k\in \{k_1,\cdots,k_j\}}} \mathbb{E}\left[\prod_{i=1}^j \left( x_{k_i} \cos(k_i X)+y_{k_i} \sin(k_i X)\right)\right]}_{=\tilde{S}_{l,k}}.
\end{eqnarray*}
\noindent
We then notice that 
\[
R_{l,k}=Q_l(\mathbf{(x,y)}_k^{(k)}), \quad \tilde{R}_{l,k}=Q_l(\mathbf{(x,y)}_{k-1}^{(k)}), \;\;\text{as well as} \;\; \mathbf{(x,y)}_{k-1}^{(k)}=\mathbf{(x,y)}_k^{(k)}.
\] 
Performing a Taylor expansion up to the order three at the points $\mathbf{(x,y)}_{k-1}^{(k)}$  and $u_{k-1}^{(k)}$,  we get

\begin{eqnarray*}
A_k&=&i t_2 \mathbb{E}\left[e^{i t_2 Q_2(\mathbf{u}_{k-1}^{(k)})+i \sum_{l=3}^p t_l Q_l \left(\mathbf{(x,y)}_{k-1}^{(k)}\right)} \frac{1}{\sigma\sqrt{n}}\left(\frac{a_k^2+b_k^2}{2}-1\right)\right]\\
&+&i\sum_{l=3}^p t_l\mathbb{E}\left[e^{i t_2 Q_2(\mathbf{u}_{k-1}^{(k)})+i \sum_{l=3}^p t_l Q_l \left(\mathbf{(x,y)}_{k-1}^{(k)}\right)} \tilde{S}_{l,k}\right]\\
&-&\frac{1}{2}\sum_{l,l'=3}^p t_l t_{l'}\mathbb{E}\left[e^{i t_2 Q_2(\mathbf{u}_{k-1}^{(k)})+i \sum_{l=3}^p t_l Q_l \left(\mathbf{(x,y)}_{k-1}^{(k)}\right)}\tilde{S}_{l,k}\tilde{S}_{l',k}\right]\\
&-&\sum_{l'=3}^p t_2 t_{l'}\mathbb{E}\left[e^{i t_2 Q_2(\mathbf{u}_{k-1}^{(k)})+i \sum_{l=3}^p t_l Q_l \left(\mathbf{(x,y)}_{k-1}^{(k)}\right)} \frac{1}{\sigma\sqrt{n}}\left(\frac{a_k^2+b_k^2}{2}-1\right)\tilde{S}_{l',k}\right]\\
&-&\frac{1}{2}t_2^2\mathbb{E}\left[e^{i t_2 Q_2(\mathbf{u}_{k-1}^{(k)})+i \sum_{l=3}^p t_l Q_l \left(\mathbf{(x,y)}_{k-1}^{(k)}\right)} \frac{1}{\sigma^2 n}\left(\frac{a_k^2+b_k^2}{2}-1\right)^2\right]\\
&+&O\left(\max_{l\in\llbracket 3,p\rrbracket} \mathbb{E}\left[|\tilde{S}_{l,k}|^3+\frac{a_k^6+b_k^6}{n\sqrt{n}}\right]\right).
\end{eqnarray*}
In a similar way we have:
\begin{eqnarray*}
B_k&=&i t_2 \mathbb{E}\left[e^{i t_2 Q_2(\mathbf{u}_{k}^{(k)})+i \sum_{l=3}^p t_l Q_l \left(\mathbf{(x,y)}_{k}^{(k)}\right)} \frac{z_k}{\sqrt{n}}\right]\\
&+&i\sum_{l=3}^p t_l\mathbb{E}\left[e^{i t_2 Q_2(\mathbf{u}_{k}^{(k)})+i \sum_{l=3}^p t_l Q_l \left(\mathbf{(x,y)}_{k}^{(k)}\right)} {S}_{l,k}\right]\\
&-&\frac{1}{2}\sum_{l,l'=3}^p t_l t_{l'}\mathbb{E}\left[e^{i t_2 Q_2(\mathbf{u}_{k}^{(k)})+i \sum_{l=3}^p t_l Q_l \left(\mathbf{(x,y)}_{k}^{(k)}\right)}{S}_{l,k}{S}_{l',k}\right]\\
&-&\sum_{l'=3}^p t_2 t_{l'}\mathbb{E}\left[e^{i t_2 Q_2(\mathbf{u}_{k}^{(k)})+i \sum_{l=3}^p t_l Q_l \left(\mathbf{(x,y)}_{k}^{(k)}\right)} \frac{z_k}{\sqrt{n}}{S}_{l',k}\right]\\
&-&\frac{1}{2}t_2^2\mathbb{E}\left[e^{i t_2 Q_2(\mathbf{u}_{k}^{(k)})+i \sum_{l=3}^p t_l Q_l \left(\mathbf{(x,y)}_{k}^{(k)}\right)} \frac{z_k^2}{n}\right]\\
&+&O\left(\max_{l\in\llbracket 3,p\rrbracket} \mathbb{E}\left[|{S}_{l,k}|^3+\frac{z_k^6}{n\sqrt{n}}\right]\right).
\end{eqnarray*}
One the one hand, since $a_k,b_k,g_k,\tilde{g_k},z_k$ and $\sigma^{-1}\left((a_k^2+b_k^2)/2-1\right)$ are centered and since the index $k$ appears in $S_{k,l}$ and $\tilde{S}_{k,l}$ and not in $\mathbf{(x,y)}_{k-1}^{(k)}=\mathbf{(x,y)}_{k}^{(k)}$  nor in $\mathbf{u}_k^{(k)}=\mathbf{u}_{k-1}^{(k)}$ by independence the first order terms in the above Taylor expansions vanish. On the other hand, as $\{a_k,b_k\}_{k\ge 1}$ have a symmetric distribution then computing the expectation with respect to terms of index $k$ gives
\begin{eqnarray*}
&&\sum_{l'=3}^p t_2 t_{l'}\mathbb{E}\left[e^{i t_2 Q_2(\mathbf{u}_{k-1}^{(k)})+i \sum_{l=3}^p t_l Q_l \left(\mathbf{(x,y)}_{k-1}^{(k)}\right)} \frac{1}{\sigma\sqrt{n}}\left(\frac{a_k^2+b_k^2}{2}-1\right)\tilde{S}_{l',k}\right]\\
&=&\sum_{l'=3}^p t_2 t_{l'}\mathbb{E}\left[e^{i t_2 Q_2(\mathbf{u}_{k}^{(k)})+i \sum_{l=3}^p t_l Q_l \left(\mathbf{(x,y)}_{k}^{(k)}\right)} \frac{z_k}{\sqrt{n}}{S}_{l',k}\right]=0.
\end{eqnarray*}
In the same way, since  $\{\frac{1}{\sigma\sqrt{n}}\left(\frac{a_k^2+b_k^2}{2}-1\right)\}_{k\ge 1}$ have variance one we get
\begin{eqnarray*}
&&\mathbb{E}\left[e^{i t_2 Q_2(\mathbf{u}_{k}^{(k)})+i \sum_{l=3}^p t_l Q_l \left(\mathbf{(x,y)}_{k}^{(k)}\right)} \frac{z_k^2}{n}\right]\\
&&=\mathbb{E}\left[e^{i t_2 Q_2(\mathbf{u}_{k-1}^{(k)})+i \sum_{l=3}^p t_l Q_l \left(\mathbf{(x,y)}_{k-1}^{(k)}\right)} \frac{1}{\sigma^2 n}\left(\frac{a_k^2+b_k^2}{2}-1\right)^2\right].\\
\end{eqnarray*}
Finally, taking into account that $\{a_k,b_k,g_k,\tilde{g}_k\}$ are all independent, centered with variance one, we may compute the expectation with terms of index $k$ as before and we get that
\begin{eqnarray*}
&&\sum_{l,l'=3}^p t_l t_{l'}\mathbb{E}\left[e^{i t_2 Q_2(\mathbf{u}_{k}^{(k)})+i \sum_{l=3}^p t_l Q_l \left(\mathbf{(x,y)}_{k}^{(k)}\right)} {S}_{l,k}{S}_{l',k}\right]\\
&=&\sum_{l,l'=3}^p t_l t_{l'}\mathbb{E}\left[e^{i t_2 Q_2(\mathbf{u}_{k-1}^{(k)})+i \sum_{l=3}^p t_l Q_l \left(\mathbf{(x,y)}_{k-1}^{(k)}\right)} \tilde{S}_{l,k}\tilde{S}_{l',k}\right].
\end{eqnarray*}
Hence, the terms of order two coincide in the Taylor expansions of $A_k$ and $B_k$ for every $k\in\{1,\cdots,n\}$. Gathering all these facts implies that
\begin{eqnarray*}
\Delta_n(t)=\sum_{k=1}^n A_k-B_k=O\left(\frac{1}{\sqrt{n}}+ \sum_{l=3}^p\sum_{k=1}^n \mathbb{E}\left[|S_{l,k}|^3+|\tilde{S}_{l,k}|^3\right]\right).
\end{eqnarray*}
Then, relying on the hypercontractivity principle given in \cite[Lemma~4.2]{nourdin2010invariance}, that we apply for the sequence of independent random variables $\mathbf{(x,y)_{k}}$ which is bounded in $L^3$ and the homogeneous sums $S_{l,k}$ and $\tilde{S}_{l,k}$, we get for some constant $C$ depending on $p$ and $\mathbb{E}\left[a_1^6\right]$
\[
\begin{array}{ll}
\displaystyle{\sum_{l=3}^p\sum_{k=1}^n \mathbb{E}\left[|S_{l,k}|^3+|\tilde{S}_{l,k}|^3\right] } & \displaystyle{\le C \sum_{l=3}^p \max_{1\le k\le n} \sqrt{\mathbb{E}\left[S_{l,k}^2\right]}\sum_{k=1}^n \mathbb{E}\left[S_{l,k}^2\right]}\\
\\
& \displaystyle{+ C \sum_{l=3}^p \max_{1\le k\le n} \sqrt{\mathbb{E}\left[\tilde{S}_{l,k}^2\right]}\sum_{k=1}^n \mathbb{E}\left[\tilde{S}_{l,k}^2\right]}.
\end{array}
\]
In order to conclude the proof, without loss of generality, we readopt the renumbering of the homogeneous sums used at the beginning of Section \ref{sec.noise}. Namely we have that 
\[
\begin{array}{ll}
\displaystyle{\max_{1\le k\le n}\sqrt{\mathbb{E}\left[S_{l,k}^2\right]}=\max_{1\le k\le n}\sqrt{\mathbb{E}\left[\tilde{S}_{l,k}^2\right]}} & \displaystyle{\le \sqrt{\text{Inf}_{2k}(f_{2n,l})+\text{Inf}_{2k-1}(f_{2n,l})}}\\
\\
& \displaystyle{\le \sqrt{2} \max_{1\le i \le 2n} \text{Inf}_i(f_{2n,l})},
\end{array}
\]
where, following the notations of \cite[e.q. (1.5)]{nourdin2010invariance}, $\text{Inf}_i(f_{2n,l})$ denotes the influence of the index $i$, namely
$$\text{Inf}_i(f_{2n,l}):=\sum_{1\le i_2<i_3\cdots<i_l\le 2n} f_{2n,l}(i,i_2,\cdots,i_l)^2.$$
As established in Section \ref{sec.noise}, in the case of Gaussian coefficients, each of the homogeneous sums $\sqrt{n}\, \mathbb{E}\left[ e_{n,j}(X)\right]$, which is induced by the symmetric kernel $f_{2n,j}$, has a Gaussian limit in distribution as $n$ goes to infinity.  Hence, relying on Equation (1.10) of \cite{nourdin2010invariance}, we know that for every $j\in \{3,\cdots,p\}$, $\max_{1\le i \le 2 n} \text{Inf}_i (f_{2n,j})$ goes to zero as $n$ goes to infinity. Finally, we also notice that
\begin{eqnarray*}
\sum_{k=1}^n \mathbb{E}\left[\tilde{S}_{l,k}^2\right]&=&\sum_{k=1}^n \mathbb{E}\left[\tilde{S}_{l,k}^2\right]\\
&=&\sum_{k=1}^{n} \sum_{\{2k,2k-1\}\cap\{i_1<i_2<\cdots,i_l\}\neq \emptyset} f_{2n,l}(i_1,i_2,\cdots,i_l)^2\\
&=& \sum_{k=1}^{2n}\sum_{k\in\{i_1<i_2<\cdots,i_l\}} f_{2n,l}(i_1,i_2,\cdots,i_l)^2\\
&=&\mathbb{E}\left[\left(\sqrt{n}\mathbb{E}_X\left[e_{n,l}(X)\right]\right)^2\right]\le \frac{1}{(l-1)!},
\end{eqnarray*}
in virtue of the computations of the $L^2$-norm of $\sqrt{n}\, \mathbb{E}_X\left[e_{n,l}(X)\right]$ which are given at the end of the Section \ref{sec.esti}. Gathering all these facts guarantees that the difference of the characteristic functions of the two random vectors $\left(Q_2(\mathbf{u}_0),Q_3(\mathbf{(x,y)}_0),\cdots,Q_p(\mathbf{(x,y)}_0)\right)$ and $\left(Q_2(\mathbf{u}_n),Q_3(\mathbf{(x,y)}_n),\cdots,Q_p(\mathbf{(x,y)}_n)\right)$ tends to zero as $n\to\infty$ which leads to the desired conclusion.


\end{document}